\documentclass[a4paper,12pt,reqno]{amsart}
\usepackage[a4paper,left=2.5cm,right=2.5cm,top=1.9cm,bottom=1.9cm]{geometry}
\usepackage{amssymb,mathrsfs,mathtools}
\usepackage[all]{xy}
\DeclareSymbolFont{script}{U}{eus}{m}{n}
\DeclareMathSymbol{\Wedge}{0}{script}{"5E}
\DeclareMathAlphabet{\mathrmsl}{OT1}{cmr}{m}{sl}
\newtheorem{theorem}{Theorem}
\newtheorem{lem}[theorem]{Lemma}
\newtheorem{prop}[theorem]{Proposition}

\theoremstyle{definition}
\newtheorem{dfn}{Definition}
\theoremstyle{remark}
\newtheorem{rk}{Remark}
\newtheorem{ex}{Example}
\renewcommand\a{\alpha}
\renewcommand\b{\beta}
\newcommand\cg{\gamma}
\renewcommand\d{\delta}

\renewcommand\l{\lambda}

\newcommand\te{\theta}
\newcommand\vp{\varphi}
\newcommand\x{\xi}
\newcommand\z{\sigma}

\newcommand\Cc{{\mathscr C}}
\newcommand\C{{\mathbb C}}
\newcommand\N{{\mathbb N}}
\newcommand\PP{{\mathbb P}}
\newcommand\R{{\mathbb R}}

\newcommand\E{{\mathcal E}}
\newcommand\cL{{\mathcal L}}
\newcommand\I{{\mathcal I}}
\newcommand\V{{\mathcal V}}
\newcommand\W{{\mathcal W}}

\newcommand\bx{{\boldsymbol x}}
\newcommand\bu{{\boldsymbol u}}
\newcommand\bF{{\boldsymbol F}}
\newcommand\dd{{\mathrmsl d}}
\renewcommand{\geq}{\geqslant}
\renewcommand{\leq}{\leqslant}
\newcommand\sub{\subseteq}
\newcommand\ot{\otimes}

\newcommand\we{\wedge}
\newcommand\p{\partial}
\newcommand{\tw}{{\mathcal T\!w}}
\newcommand{\chv}{\chi}
\newcommand{\cch}{{\mathcal C}}
\newcommand{\ccov}{{\mathcal Q}}
\newcommand{\normeq}{{\mathfrak e}}
\newcommand{\Id}{\mathrmsl{Id}}
\newcommand\op[1]{\mathop{\mathrm{#1}}\nolimits}
\newcommand\com[1]{}
\newcommand\inm{\mathbin{\!\lrcorner}}
\newcounter{numl}
\newcommand{\labelnuml}{\textup{(\roman{numl})}}
\newenvironment{numlist}{\begin{list}{\labelnuml}%
{\usecounter{numl}\setlength{\leftmargin}{0pt}%
\setlength{\itemindent}{2\parindent}%
\setlength{\itemsep}{\smallskipamount}\def
\makelabel ##1{\hss \llap {\upshape ##1}}}}{\end{list}}

\begin{document}

\title[Integrability via geometry]{Integrability via geometry: dispersionless
differential equations in three and four dimensions}
\author{David M.J. Calderbank}
\address{DC: Department of Mathematical Sciences, University of Bath,
Claverton Down, Bath, BA2 7AY, UK. \quad
\emph{Email:} D.M.J.Calderbank@bath.ac.uk}
\author{Boris Kruglikov}
\address{BK: Institute of Mathematics and Statistics, UiT the Arctic University of Norway,
Troms\o\ 90-37, Norway. \quad \emph{Email:} Boris.Kruglikov@uit.no}
\date{}
\begin{abstract}
We prove that the existence of a dispersionless Lax pair with spectral
parameter for a nondegenerate hyperbolic second order partial differential
equation (PDE) is equivalent to the canonical conformal structure defined by
the symbol being Einstein--Weyl on any solution in 3D, and self-dual on any
solution in 4D. The first main ingredient in the proof is a characteristic
property for dispersionless Lax pairs. The second is the projective behaviour
of the Lax pair with respect to the spectral parameter. Both are established
for nondegenerate determined systems of PDEs of any order.  Thus our main
result applies more generally to any such PDE system whose characteristic
variety is a quadric hypersurface.
\end{abstract}
\maketitle
\vspace{-5mm}

\section*{Introduction and main results}

The integrability of dispersionless partial differential equations is well
known to admit a geometric interpretation.  Twistor theory~\cite{Pe,MW}
gives a framework to visualize this for several types of integrable systems, as
demonstrated by many examples~\cite{Pb,Hi,DMT,Wa,C1}.

Recently, such a relation has been established for several classes of second
order equations in 3D and one class in 4D \cite{FK}.  Namely the following
equivalences have been established:
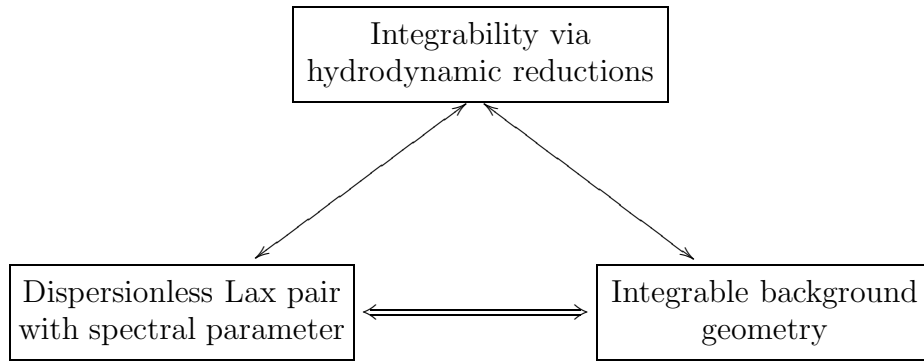
\vspace{-0.4cm}
\[
\vspace{-0.1cm}
\xymatrix{&
\hspace{-3cm}\parbox{3.8cm}{\fbox{\parbox{4.7cm}{\begin{center}Integrability via
\\ hydrodynamic reductions\end{center}}}
                           \\[1.1cm] }
   \ar@{<->}[ld] \ar@{<->}[rd] \hspace{-2cm} & \\
\fbox{\parbox{4.3cm}{\begin{center}Dispersionless Lax pair \\
with spectral parameter\end{center}}}
   \ar@{<=>}[rr] & \hspace{1cm} &
\fbox{\parbox{4.2cm}{\begin{center}Integrable background \\
geometry\end{center}}}}
\]

Hydrodynamic integrability in 2D (also written ``$1+1$ dimension'') was
introduced in \cite{DN} and elaborated in \cite{Ts}. Integrability via
hydrodynamic reductions in $d\geq 3$ dimensions was developed in \cite{FKh}.
This method, although constructive, is not universal, as it applies only to
translation invariant equations (invariantly, this requires the existence 
of a $d$-dimensional abelian contact symmetry group). Thus the upper part 
of the above diagram, at least at present, does not extend to 
the general class of second order PDEs.

On the other hand, the two other ingredients of the diagram are universal.
The main aim of this paper is to prove the bottom equivalence for large class
of PDE systems, including general second order PDEs, in 3D and 4D, where
``integrable background geometry'' means that a canonical conformal structure
on solutions of the equation is Einstein--Weyl in 3D and self-dual in 4D
(these geometries are ``backgrounds'' for integrable gauge
theories~\cite{C1}).

Consider a second order PDE
\begin{equation}\label{F}
\E\;:\; F(\bx,u,\p u, \p^2 u)=0
\end{equation}
for a scalar function $u$ of an independent variable $\bx$ on a connected
manifold $M$ with $\dim M=d$, where $\p u=(u_i)$ and $\p^2u=(u_{ij})$ denote
partial derivatives of $u$ in local coordinates $\bx=(x^i)$. Let $M_u$ denote
the manifold $M$ equipped with a given scalar function $u$; concretely, we may
view $M_u$ as the graph of $u$ in $M\times\R$. A tensor on $M_u$ is, by
definition, a tensor on $M$, which may also depend, at each $\bx\in M$, on
finitely many derivatives of $u$ at $\bx$.

Let $\z_F$ be the linearization of $F$ in second derivatives, i.e.,
\begin{equation*}
\z_F = \sum_{i\leq j}\frac{\p F}{\p u_{ij}}\,\p_i\p_j=
\sum_{i,j} \z_{ij}(u)\, \p_i\otimes \p_j, \quad\text{where} \quad
\z_{ij}(u):=\frac{1+\d_{ij}}{2}\frac{\p F}{\p u_{ij}}.
\end{equation*}
Invariantly, $\z_F$ defines a section of $S^2T M_u$, hence a quadratic form on
$T^*_\bx M_u$ for each $\bx\in M_u$, called the \emph{symbol} of $F$. 
If we change the defining function $F$ of $\E$,
$\z_F$ changes by a conformal rescaling on $\E$. Hence the conformal class of $\z_F$ 
along $F=0$ is an invariant of~$\E$, as is the \emph{characteristic variety}
$\chv^\E\to M_u$, the bundle whose fibre at $\bx\in M_u$ is the projective variety
$\chv^\E_\bx:=\op{Char}(\E,u)_{\bx}=\{[\te]\in \PP(T^*_\bx M_u)\,|\, \z_F(\te)=0\}$.  

We assume henceforth that~\eqref{F} is:
\begin{itemize}
\item \emph{nondegenerate}, i.e., $\z_F$ is nondegenerate at generic points of
  the zero-set $\E$ of $F$.  This is equivalent to $\det(\z_{ij}(u))\neq0$ for a
  generic solution $u$.
\item \emph{hyperbolic}, i.e., $M$ is complex and $F$ is holomorphic, or $M$
  is real, $F$ is smooth and the variety $\{[\te]\in \PP(T^*M_u\ot\C)\,|\,
  \z_F(\te)=0\}$ of complex characteristics is a complexification of $\chv^\E$
  for a generic (real) solution $u$.
\end{itemize}
The nondegeneracy of $\z_F$ implies that its inverse
\[
g_F=\sum_{i,j} g_{ij}(u) \,\dd x^i\,\dd x^j,\ \text{ where }
(g_{ij}(u))=(\z_{ij}(u))^{-1},
\]
defines a nondegenerate symmetric bilinear form on $T_\bx M_u$ for any
$(\bx,u)$ sufficiently close to a generic point of $F=0$. As in~\cite{FK}, the
corresponding conformal structure $c_F$ plays a central role in this paper.
Hyperbolicity implies that along $F=0$, $c_F$ is uniquely determined by the
bundle $\chv^\E$ of nonsingular quadric hypersurfaces because the latter is dual 
to the projectivized null cone of $c_F$.

A \emph{dispersionless Lax pair}~\cite{Za} or \emph{dLp} for~\eqref{F} can be
described as rank one covering system~\cite{Vi} of $\E$. Roughly speaking,
this means that there is a fibre bundle $\hat\pi\colon \hat M_u\to M_u$ with
connected rank one fibres, and a PDE system on $\hat M_u$ with $\E$ as a
differential corollary. There are various ways to formulate this precisely;
in this paper we adopt as a definition that there are linearly independent vector
fields $\hat X$ and $\hat Y$ on $\hat M_u$, whose coefficients depend on finitely many
derivatives of $u$, such that $\E$ is the Frobenius integrability condition
for their span $\hat\Pi\sub T\hat M_u$---this is the condition that $[\hat
  X,\hat Y]$ is a section of $\hat\Pi$, so that $\hat\Pi$ is tangent to a
foliation of $\hat M_u$ by surfaces.

The leaf space of this foliation (for a solution $u$ of $\E$) is sometimes
called the \emph{twistor space} $\tw$ of the dLp in 4D (or \emph{minitwistor space}
in 3D). However, a well-behaved twistor space may only exist over suitable
open subsets of $M_u$, so its geometry is more conveniently described on the
\emph{correspondence space} $\hat M_u$.  For instance, functions on $\tw$ 
correspond to solutions of a linear PDE system for functions on $\hat M_u$
that are constant on the leaves of the foliation, while hypersurfaces in
$\tw$ may be described as solutions of a quasilinear PDE system for sections
of $\hat\pi\colon \hat M_u\to M_u$ that are unions of such leaves. Either of
these PDE systems can equivalently be called a dLp: $\E$ ensures their
compatibility.

A fibre coordinate $\l\colon\hat M_u\to \R$ is called a \emph{spectral
  parameter} and it locally identifies $\hat M_u$ with $M_u\times\R$. We may then
write $\hat X=X+m\,\p_\l$, $\hat Y=Y+n\,\p_\l$ where $X,Y$ are $\l$-parametric
vector fields on $M_u$, and a section of $\hat\pi\colon\hat M_u\to M_u$ may be
written $\l=q(\bx)$ for a function $q\colon M_u\to\R$. The dLp $\hat\Pi$ then
has the geometric interpretation that $\E$ is the integrability condition for
the existence of many foliations of $M_u$ by surfaces which are tangent at any
$\bx\in M_u$ to the span $\Pi=\hat\pi_*(\hat\Pi)$ of $X$ and $Y$ at
$\bx$, with $\l=q(\bx)$.

A fundamental motivation for this paper is that in all known examples of such
dLps, it has been observed (see e.g.~\cite{FK}) that $\Pi$ is
\emph{characteristic} for $\E$ in the sense that for any $1$-form $\te$ on
$M_u$ with $\Pi\sub\ker\te$, we have $[\te]\in \chv^\E$. Thus for any
solution $u$ of $\E$, $M_u$ admits many foliations by \emph{characteristic
  surfaces}, and indeed $\E$ is the integrability condition for their
existence.  Our first result establishes this characteristic property in
considerable generality.

\begin{theorem}\label{t:char} Let $\hat{\Pi}$ be a dLp
on $\hat\pi\colon\hat M_u\to M_u$ for a determined PDE system $\E$ of order
$\ell$ on $M_u$. Then $\Pi=\hat\pi_*(\hat\Pi)$ is characteristic for $\E$.
\end{theorem}

We refer to Sections~\ref{Sec:1} and~\ref{Sec:2}, or~\cite{KV,KL,Vi}, for
discussion of more general PDE systems and their characteristic varieties: in
this introduction, we focus on second order scalar PDEs. For such PDEs, the
characteristic condition means that for each solution $u$ and each
$\hat\bx\in\hat M_u$, $\Pi_{\hat\bx}$ is a coisotropic $2$-plane for the
conformal structure $c_F$.  By nondegeneracy of $c_F$, such $2$-planes can
only exist for $2\leq d\leq 4$: for $d=2$, the condition is vacuous; for
$d=3$, $\Pi_{\hat\bx}$ is then tangent to the null cone of $c_F$ (i.e.,
degenerate); for $d=4$, $\Pi_{\hat\bx}$ is then contained in the null cone
(i.e., totally isotropic). In the real case, the characteristic condition
further implies that $c_F$ has (up to sign) signature $(2,1)$ for $d=3$ or
$(2,2)$ for $d=4$.  We assume this henceforth.

For both $d=3$ and $d=4$, the coisotropic $2$-planes at each point $\bx\in M$
form a $1$-dimensional submanifold of the grassmannian $\op{Gr}_2(T_{\bx}M)$.
For $d=3$ this submanifold is a rational curve ($\cong\PP^1$, the projective
line) canonically isomorphic to the conic $\chv^\E\sub\PP(T^*_{\bx}M)$.  For
$d=4$, it is a disjoint union of two rational curves, corresponding to the two
rulings of the quadric surface $\chv^\E$; the points of the two components are
called $\a$-planes and $\b$-planes 
depending on whether the $2$-planes are self-dual or anti-self-dual.

If $\Pi$ is coisotropic and is also an immersion, we may thus identify $\hat
M_u$ locally with the $\PP^1$-bundle whose fibre over $\bx\in M_u$ consists of
all coisotropic $2$-planes for $d=3$ or the $\a$-plane component for
$d=4$. Under this identification, $\Pi\to \hat M_u$ becomes the tautological
bundle of coisotropic $2$-planes.  Any \emph{Weyl connection} $\nabla$ on
$M_u$ (a torsion-free conformal connection on $M$ depending on finitely many
derivatives of $u$) induces a connection on $\hat M_u\to M_u$ and hence a
horizontal lift of $\Pi$ to distribution $\hat\Pi_\nabla\sub T\hat M_u$.

If $d=4$, it is well-known~\cite{Pe} that $\hat\Pi_\nabla$ is independent of
$\nabla$ (i.e., \emph{conformally invariant}), and is integrable if and only
if $(M_u,c_F)$ is is \emph{self-dual} (SD), i.e., the Weyl tensor $W_{c_F}$
satisfies $W_{c_F}=*W_{c_F}$. The integral surfaces of $\hat\Pi_\nabla$ then
project to $\alpha$-surfaces for $c_F$.

If $d=3$, it is similarly well known~\cite{Cartan,Hi} that $\hat\Pi_\nabla$
is integrable if and only if $(M_u,c_F,\nabla)$ is \emph{Einstein--Weyl} (EW),
i.e., the symmetrized Ricci tensor of $\nabla$ is proportional to any metric
$g_F$ in the conformal class: $\op{Sym}(\op{Ric}^\nabla)=\Lambda\, g_F$,
$\Lambda\in C^\infty(M_u)$. The integral surfaces of $\hat\Pi_\nabla$ then
project to totally geodesic null surfaces for $(c_F,\nabla)$.

A dLp $\hat\Pi$ for $\E$ arising in this way for $d=3,4$ will be called \emph{standard}.
Two dispersionless Lax pairs $\hat\Pi$, $\hat\Pi'$ will be called \emph{$\E$-equivalent}, 
if $\hat\Pi=\hat\Pi'$ on $\hat M_u$ for any solution $u$ of $\E$.

It is an open question in the theory of integrable systems how many non-equivalent coverings
a given $\E$ can possess. Our second result claims that coverings of dLp type
are essentially unique under a certain nondegeneracy condition on $\hat\Pi$. 
This condition, given in Definition \ref{Dnd} of Section \ref{Sec:33}, depends
only on $\Pi=\hat\pi_*(\Pi)$, implies that $\Pi$ immerses, and holds in all
examples we know of.

The result is straightforward when $d=4$, but when $d=3$, it shows that $\hat\Pi$
can be assumed \emph{projective}: for some choice of spectral parameter
$\l$ and vector fields $\hat{X},\hat{Y}$ generating $\hat{\Pi}$, the
coefficients of these vector fields are cubic polynomials in $\l$. The result
is again not restricted to second order scalar PDEs: we require only that
$\chv^\E_\bx$ is a nonsingular quadric hypersurface for each $\bx\in M_u$.

\begin{theorem}\label{t:proj} 
Let $\E: F=0$ be a determined PDE system of order $\ell$ whose characteristic
variety $\chv^\E$ is a bundle of nonsingular quadric hypersurfaces in $\PP(T^*M_u)$.  
Then any nondegenerate dLp $\hat\Pi$ is $\E$-equivalent to a standard dLp $\hat\Pi_\nabla$ 
for some Weyl connection $\nabla$.
\end{theorem}

Our third (and main) result establishes an equivalence between the
dispersionless integrability of $\E$ and the EW/SD property of $c_F$. However,
to achieve this, some care is needed in the formulation of both properties.
First, in the integrability of the dLp $\hat\Pi$, we must account for
$\E$-equivalence. Thus we say that $\E$ is \emph{integrable by a dLp
  $\hat\Pi$} if for any $\hat\Pi'$, which is $\E$-equivalent to $\hat\Pi$, the
Frobenius integrability condition for $\hat\Pi'$ is a nontrivial differential
corollary of $\E$. Secondly, the EW/SD property should be a nontrivial
differential corollary of $\E$. The need for nontriviality here is illustrated
by PDEs of the form $\Delta u=f(\bx,u,\p u)$: this is non-integrable for
generic $f$, but its conformal structure is independent of $u$ and is flat, so
the EW/SD property holds automatically. For more general PDE systems $\E$, a
differential corollary of $\E$ \emph{holds nontrivially} if it is not a
consequence of a proper subsystem $\E'$ of $\E$. 
We can now obtain the main result as follows.

\begin{theorem}\label{t:main}
Let $\E:F=0$ be a determined PDE system in \textup{3D} or \textup{4D} whose
characteristic variety $\chv^\E$ is a bundle of nonsingular quadric
hypersurfaces, for instance a nondegenerate hyperbolic second order scalar PDE
\eqref{F}. Let $c_F$ be the corresponding conformal structure.  Then $\E$ is
integrable by a nondegenerate dLp if and only if
\begin{itemize}
\item[3D:] the Einstein--Weyl property for $c_F$ holds nontrivially on
  solutions of $\E$\textup;
\item[4D:] the self-duality property for $c_F$ holds nontrivially on
  solutions of $\E$.
\end{itemize}
\end{theorem}
\begin{proof} As a preliminary, note that if $F$ has order $\ell$, then $c_F$
depends pointwise only on derivatives of $u$ of order $\leq\ell$ (or
$\leq(\ell-1)$ if $F$ is quasilinear) and so is defined and is nondegenerate 
for almost any $u$ (not necessarily a solution). Thus $\hat\Pi_\nabla$ is defined 
for any Weyl connection $\nabla$ over an open subset of $M_u$, and its 
integrability there is equivalent to the EW condition for
$(c_F,\nabla)$ when $d=3$ and the SD condition for $c_F$ when $d=4$.

Suppose first that $\hat\Pi\sub T\hat M_u$ is a dLp for $\E$. By
Theorem~\ref{t:char}, $\Pi=\hat\pi_*(\hat\Pi)$ is characteristic, i.e., when
$F=0$, $\Pi$ is coisotropic for the conformal structure $c_F$ (and for $d=4$
we orient $M_u$ so that $\Pi$ is a congruence of $\alpha$-planes).
Nondegeneracy of $\hat\Pi$ implies that $\Pi$ immerses into $\op{Gr}_2(TM_u)$
and so we may assume that $\hat M_u$ is an open subset of the bundle of
coisotropic $2$-planes for all solutions $u$, and hence also on an open
neighbourhood of $(\bx,u)$ where $c_F$ is nondegenerate. Then by Theorem
\ref{t:proj}, $\hat\Pi$ is $\E$-equivalent to a standard dLp $\hat\Pi_\nabla$
over any open subset of $M_u$. Hence the EW/SD condition is a nontrivial
differential corollary of $\E$, as required.

Conversely, suppose that the EW/SD condition is a nontrivial differential
corollary of $\E$ (for some Weyl connection $\nabla$ when $d=3$), and let
$\hat\pi\colon \hat M_u\to M_u$ be the bundle of null $2$-planes for $d=3$, or
the bundle of $\alpha$-planes for $d=4$. Then if $\hat\Pi$ is $\E$-equivalent
to $\hat\Pi_\nabla$ (for any Weyl connection $\nabla$ when $d=4$) on an open
subset of $M_u$, the integrability of $\hat\Pi$ is a differential corollary of
$\E$ on that open subset (since this is true for $\hat\Pi_\nabla$). 

Finally if any such $\hat\Pi$ is a differential corollary of a proper subsystem 
$\E'$ of $\E$, then the first part of the argument implies that the EW/SD property 
is also a consequence of $\E'$, contradicting nontriviality.
\end{proof}

\begin{rk}
Often, in the physics literature, little distinction is made between a system $\E$ and
a system $\E'$ obtained by differentiation or potentiation of $\E$. While some properties of the
equation can change, for instance the symmetry algebra and dimension of the solution space,
the characteristic variety and integrability of $\E$ are unaltered.
It is easy to adjust the formulation of the theorems to such variations between $\E$ and $\E'$.
\end{rk} 

This theorem shows that the EW and SD equations are master equations, in 3D
and 4D respectively, for determined integrable PDE systems whose
characteristic variety is a bundle of nonsingular quadric hypersurfaces.  It
applies in particular to first order systems and higher order scalar equations
whose (principal) symbol is a power of a nondegenerate quadratic form.
However, the EW and SD equations are not themselves determined systems because
of the gauge freedom coming from diffeomorphism invariance. Determined forms
of the EW and SD equations were derived in~\cite{DFK}, where it was shown in
particular that the Manakov--Santini system~\cite{MS} is equivalent to a
determined form of the EW equation. Because of their importance, we will present
novel derivations of these determined master equations using the methods of
this paper.

\smallbreak

Theorem~\ref{t:main} is useful for at least two reasons. First, the geometric
characterizations of integrability are algorithmic. In 4D, the anti-self-dual
part of the Weyl tensor of $c_F$ on $M_u$ can be computed explicitly from
finitely many derivatives of $u$, and so we can check whether it vanishes on
solutions by imposing the equation and its prolongations formally---we do not
have to be able to resolve the PDE or even to prove its solvability. In 3D,
the situation is complicated slightly by the choice of Weyl connection. For
the classes of translation-invariant equations considered in \cite{FK},
there is a universal formula for the Weyl connection, but this formula is not
generally applicable (it is not contact-invariant).  Nevertheless, except in
degenerate situations, the choice is uniquely determined by finitely many
derivatives of $c_F$, and so the EW condition may again be verified by
formally imposing the PDE on a tensor depending on finitely many derivatives
of $u$. This effective integrability criterion has many applications: for
instance, it was applied in \cite{KM} to obtain infinitely many new integrable
equations in 4D as deformations of integrable Monge-Amp\`ere equations of
Hirota type.

Secondly, the EW/SD property provides a canonical characteristic Lax pair,
which, if the PDE on $u$ has order $\ell$, depends on at most $\ell+1$
derivatives of $u$ ($\ell$ if the PDE is quasilinear), and satisfies a
`normality' condition off shell which is useful in computations.  None of
these properties were assumed a priori. For example, the standard Lax
pair~\cite{MS} for the Manakov--Santini system is not normal, and the normal
Lax pair may be understood as a Lax pair for an equivalent PDE system
presented in~\cite{DFK}, which we also discuss.

Apart from the Manakov--Santini system (and variants), Theorem~\ref{t:main}
encompasses many examples in 3D, such as the Lax pairs arising in the central
quadric ansatz~\cite{FHZ}, for EW manifolds in diagonal
coordinates~\cite{DFK}, and for the systems of two first order PDE on two
unknown functions studied in~\cite{DFKN1}. In 4D, there are Lax pairs having
no derivatives with respect to the spectral parameter $\l$, which cannot
be normal, such as the hypercomplex Lax pair of Dunajski and Joyce
(see~\cite{C1,DFK}) and Lax pairs for Monge-Amp\`ere equations of Hirota
type~\cite{DF}. However, normal Lax pairs are always available, and provide a
canonical choice in 4D, while in 3D they are given by a choice of Weyl
connection.

\smallbreak

We begin the body of the paper in Section~\ref{Sec:1} by presenting a rigorous
definition of what should be called a (nondegenerate) dispersionless Lax pair,
motivated by examples. The search for such formalism in general has a long
history: see \cite{CN,Mar} for discussion in the dispersive context. A
fundamental role is played by the $\l$-dependent family $\Pi=\hat\pi_*(\hat\Pi)$ 
of rank $2$ subbundles of $TM_u$, which we call a \emph{$2$-plane congruence}. 
We also explain the normality condition
mentioned above, observing that in 4D it determines $\hat\Pi$ from $\Pi$.

In Section~\ref{Sec:2}, we prove Theorem~\ref{t:char}.  Here we treat the
symbol and characteristic variety of general PDE systems. 
For both this, and the proof of Theorem~\ref{t:proj}, we require some
jet theory, which we have generally suppressed in the rest of the paper,
cf.~Remark~\ref{r:jet}. Having proven Theorem~\ref{t:char}, as an addendum, we
show in Section~\ref{S:Q} that a Lax pair which is characteristic for a
quadric is nondegenerate, and give a computational criterion for the existence
of such a quadric for nondegenerate Lax pairs.

For PDE systems whose characteristic variety is a quadric,
Theorem~\ref{t:char} shows that $\Pi$ is essentially unique, which
considerably constrains the choice of $\hat\Pi$, especially in 4D. In 3D,
however, more work is needed to prove Theorem~\ref{t:proj}, which we develop
in Section~\ref{Sec:3}. We first discuss the standard EW/SD Lax pairs, which
are not only normal, but projective.  We also introduce and motivate a
stronger nondegeneracy condition on the Lax pair $\hat\Pi$. Roughly speaking,
this condition means that the equation $\E$ appears nontrivially in the symbol
of the integrability condition for $\hat\Pi$ (i.e., at highest order). From
this we deduce the projective property of the Lax pair, and hence prove
Theorem~\ref{t:proj}.

In Section~\ref{Sec:4} we discuss applications and extensions of the viewpoint
we have developed. In particular, we discuss pseudopotentials and their
relation to contact coverings, the twistor interpretation of this
relationship, and potential generalizations of the theory.

\section{Lax pairs: nondegeneracy and normalization}\label{Sec:1}

\subsection{Dispersionless pairs and $2$-plane congruences}

We begin with a well-known prototypical example.
\begin{ex}[dKP] The dispersionless Kadomtsev--Petviashvilli (dKP) equation
(see for example~\cite{DMT}) is the second order scalar PDE
\begin{equation}\label{dKP}
F(\bx,u,\p u, \p^2u):=u_{xt}+(uu_t)_t-u_{yy}=0
\end{equation}
for a scalar function $u$ on a $3$-manifold $M_u\simeq M$ with coordinates
$(x,y,t)$. (This differs from some standard conventions by the interchange
$t\leftrightarrow x$ and/or $u\mapsto-u$.)  The dKP equation is the
compatibility condition $\psi_{xy}=\psi_{yx}$ of the first order linear system
\[
\psi_x - (\l^2 - u)\, \psi_t - (u_y+\l u_t)\,\psi_\l=0, \qquad
\psi_y - \l\, \psi_t - u_t\,\psi_\l=0,
\]
for a scalar function $\psi$ on $\hat M_u= M_u\times\R$ with coordinates
$(x,y,t,\l)$. It may also be described as the compatibility condition
$q_{xy}=q_{yx}$ for the quasilinear system
\[
q_x=(q^2-u)\,q_t-q\,u_t-u_y,\qquad q_y=q\,q_t-u_t
\]
for a scalar function $q=q(x,y,t)$ on $M_u$. In more geometric terms, $\psi$
is a function on $\hat M_u$ which is invariant under the vector fields
\begin{equation}\label{dKPdlp}
\hat X = \p_x - (\l^2 - u)\, \p_t - (u_y+\l u_t)\,\p_\l, \qquad
\hat Y = \p_y - \l\, \p_t - u_t\,\p_\l,
\end{equation}
while $q$ defines a section of $\hat\pi\colon \hat M_u\to M_u$ such that $\hat
X$ and $\hat Y$ are tangent to its image. The compatibility condition in
either case is that $\hat X$ and $\hat Y$ span a distribution $\hat\Pi\sub
T\hat M_u$ which is (Frobenius) integrable, i.e., $[\hat X,\hat Y]$ is also
section of $\hat\Pi$.  In this example, the Frobenius integrability condition
holds if and only if $[\hat X,\hat Y]=0$ if and only if~\eqref{dKP} is
satisfied.
\end{ex}

In this paper, we take the distribution $\hat\Pi$ on $\hat M_u$ to be the
fundamental object.
\begin{dfn}\label{DLP} A \emph{dispersionless pair} of order $\leq N$
is a bundle $\hat\pi\colon \hat M_u\to M_u$ called the \emph{correspondence
  space}, whose fibres are connected curves, together with a rank two
distribution $\hat\Pi\sub T\hat M_u$ such that:
\begin{itemize}
\item for all $\hat\bx\in \hat M_u$, $\hat\Pi_{\hat\bx}\sub T_{\hat\bx} \hat
  M_u$ depends on $u$ only through its partial derivatives at
  $\bx=\hat\pi(\hat\bx)\in M_u$ of order $\leq N$;
\item $\hat\Pi$ is transverse to the fibres of $\hat\pi$, i.e.,
$\hat\Pi\cap\ker\hat\pi_*=0$.
\end{itemize}
A \emph{spectral parameter} is a local fibre coordinate $\l=\l(\hat\bx)\colon
\hat M_u\to\R$ on $\hat M_u$.
\end{dfn}
If $\hat\Pi=\langle \hat X,\hat Y\rangle$, we thus obtain a first order
linear system
\begin{equation}\label{eq:psi}
\hat{X}(\psi)=0,\qquad \hat{Y}(\psi)=0
\end{equation}
for functions $\psi$ on $\hat M_u$. In terms of a spectral parameter $\l$, a
section of $\hat\pi$ has image $\l=q(\bx)$ for a function $q\colon M_u\to\R$,
and the corresponding first order quasilinear system is given by
\begin{equation}\label{eqQ}
\hat{X}(\l-q(\bx))|_{\l=q(\bx)}=0,\qquad
\hat{Y}(\l-q(\bx))|_{\l=q(\bx)}=0.
\end{equation}
The system~\eqref{eq:psi} is compatible if and only if~\eqref{eqQ} is
compatible if and only if the distribution $\hat\Pi$ is integrable. Then
solutions of~\eqref{eq:psi} and \eqref{eqQ} describe respectively functions
and hypersurfaces in the (local) leaf space of the folation tangent to
$\hat\Pi$ (the twistor or minitwistor space). The integrability condition of
$\hat\Pi$ is a PDE on $u$ of order $\leq N+1$. Roughly speaking---see
Definition~\ref{DIS}---dispersionless integrable systems are PDEs arising as
such integrability conditions.

We need not restrict attention to scalar PDEs. Indeed we wish to encompass the
following important system due to Manakov and Santini~\cite{MS}.

\begin{ex}[MS]\label{ex:MS} The Manakov--Santini (MS) system is the second
order coupled system of PDEs
\begin{equation}\label{MS}
  S(u)+u_t^2=0, \qquad S(v)=0
\end{equation}
for functions $(u,v)$ of $(x,y,t)$, where
\begin{equation}\label{eq:Sdef}
  S = \p_t\p_x+v_t\, \p_t\p_y+(u-v_y)\, \p_t^2-\p_y^2.
\end{equation}
(As with the dKP equation, we have aligned our coordinate conventions for
consistency within this paper.  Conventions in the
literature~\cite{MS,DFK,PCC} vary, but are all equivalent to the one here by
point transformations.)

As noted in~\cite{MS}, system~\eqref{MS} is the Frobenius integrability
condition for the dispersionless pair $\hat\Pi=\langle \hat X,\hat Y\rangle$
spanned by
\begin{equation}\label{LP-MS}
\hat X=\p_x-(\l^2+v_t\l-u+v_y)\,\p_t
-(u_t\l+u_y)\,\p_\l, \qquad
\hat Y=\p_y-(\l+v_t)\,\p_t-u_t\,\p_\l.
\end{equation}
The corresponding quasilinear covering system, which was studied in
\cite{PCC} and more recently in~\cite{Pr}, is
\[
q_x=(q^2+q\,v_t-u+v_y)\,q_t-q\,u_t-u_y,\qquad
q_y=(v_t+q)\,q_t-u_t.
\]
When $v=0$, the MS system reduces to the dKP equation, and~\eqref{LP-MS}
to~\eqref{dKPdlp}. When $u=0$, the dLp~\eqref{LP-MS} has no derivatives with
respect to the spectral parameter.
\end{ex}

If $\hat\Pi$ is a dispersionless pair, then $\Pi:= \hat\pi_*(\hat\Pi)$ is a
rank $2$ subbundle of $\hat\pi^*TM_u$, so at each $\bx\in M_u$, we have a
$1$-parameter family of $2$-dimensional subspaces of $T_\bx M$.
\begin{dfn}
A \emph{$2$-plane congruence} $\Pi$ over $M_u$ is a section $\Pi\colon \hat
M_u\to\hat\pi^*\op{Gr}_2(TM_u)$, where $\op{Gr}_2(TM_u)\to M_u$ is the bundle
whose fibre over $\bx\in M_u$ is the grassmannian of $2$-dimensional vector
subspaces of $T_\bx M_u$.
\end{dfn}
Conversely, the passage from a $2$-plane congruence $\Pi$ to a dispersionless
pair $\hat\Pi$ can be understood as a \emph{lift} with respect to the projection 
$\hat\pi\colon \hat M_u\to M_u$.  It is convenient to describe the lift condition 
in terms of the rank $3$ distribution $\Delta=\hat\pi_*^{-1}(\Pi)\sub T\hat M_u$:
$\hat\Pi$ is a lift of $\Pi$ if and only if it
is a rank $2$ subbundle of $\Delta$ transverse to the fibres of $\hat\pi$. For
any distributions $D_1,D_2\sub T\hat M_u$ we denote by $[D_1,D_2]$ the
distribution generated by Lie brackets of sections of $D_1$ and $D_2$. Thus
the integrability condition for $\hat\Pi$ is that its derived distribution
$[\hat\Pi,\hat\Pi]$ is equal to $\hat\Pi$.

More explicitly, we choose a spectral parameter $\l$ and let $X,Y$ be
linearly independent $\l$-parametric vector fields on $M_u$ depending at each $\bx$
only on the partial derivatives of $u$ at $\bx$ of order $\leq N$. Then
$\Pi=\langle X,Y\rangle$ is a $2$-plane congruence, and $\Delta$ is the span
of the coordinate lifts of $X,Y$ (still denoted $X,Y$, with
$X(\l)=0=Y(\l)$) and $\p_\l$. Then we write a dispersionless pair
$\hat\Pi$ on $\hat M_u$, with $\hat\pi_*(\hat\Pi)=\Pi$ as the span
$\hat\Pi=\langle \hat X,\hat Y\rangle$ of vector fields
\begin{equation}\label{XYhat}
\hat X=X+m\,\p_\l,\qquad \hat Y=Y+n\,\p_\l
\end{equation}
with $\hat\pi_*(\hat X)=X$ and $\hat\pi_*(\hat Y)=Y$, where $m,n$ are
functions of $\bx$, $u$, and the spectral parameter $\l$.  The derived
distribution of $\hat\Pi$ is now $[\hat\Pi,\hat\Pi]=\langle\hat X,\hat
Y,[\hat X,\hat Y]\rangle\sub T\hat M_u$, which generically has rank
$3$, and the integrability condition is that it has rank $2$.

In 3D, we may introduce coordinates $(x,y,t)$ and choose generators of $\Pi$
of the form
\begin{equation}\label{XY3}
X=\p_x-\a\, \p_t,\qquad Y=\p_y-\b\, \p_t,
\end{equation}
where the functions $\a,\b$ depend on $(x,y,t)$, $u$ and $\l$. Dually,
the annihilator $\op{Ann}(\Pi)$ of $\Pi$ in $\hat\pi^* T^*M_u$ is spanned
by the $\l$-dependent $1$-form
\begin{equation}\label{thet}
\te= \dd t + \a\,\dd x + \b\, \dd y,
\end{equation}
$\op{Ann}(\Delta)$ is spanned by the pullback of $\te$ to $\hat
M_u$ (which we still denote by $\te$), while $\op{Ann}(\hat\Pi)$ is
spanned by $\te$ and the $1$-form
\begin{equation}\label{eq:eta}
\eta= \dd\l - m\, \dd x - n\, \dd y
\end{equation}
on $\hat M_u$. Hence $\hat\Pi$ is the radical of the $2$-form
$\te\wedge\eta$.

In 4D, we similarly may assume that we have coordinates $(x,y,z,t)$ and
generators
\begin{equation}\label{XY4}
X=\p_x-\a\, \p_z-\b\, \p_t,\qquad Y=\p_y-\cg\, \p_z-\d\, \p_t,
\end{equation}
where $\a,\b,\cg,\d$ depend on $(x,y,z,t)$, $u$ and $\l$. Thus $\op{Ann}(\Pi)$
is spanned by
\begin{equation}\label{zeth}
  \zeta= \dd z + \a\, \dd x+ \cg\,\dd y,\qquad
  \te= \dd t + \b\,\dd x + \d\, \dd y,
\end{equation}
$\op{Ann}(\Delta)$ by their pullbacks, and $\op{Ann}(\hat\Pi)=\langle
\zeta,\te,\eta\rangle$ with $\eta$ given by~\eqref{eq:eta}.  In both 3D and
4D, with $\hat X$ and $\hat Y$ given by~\eqref{XYhat}, 
$\hat\Pi$ is integrable if and only if $[\hat X,\hat Y]=0$.

\subsection{Normality and nondegeneracy}

In order for $\hat\Pi$ to be a dispersionless Lax pair for an equation $\E:
F=0$, we require that the integrability condition $[\hat\Pi,\hat\Pi]=\hat\Pi$
holds modulo $\E$, i.e., when $F=0$ or, to use physics terminology, \emph{on
  shell}.

\begin{dfn} We say that the dispersionless pair $\hat\Pi\sub T\hat M_u$
is \emph{normal} if $[\hat\Pi,\hat\Pi]\sub \Delta$ off shell, i.e., without
assuming $F=0$. In other words, $\hat\pi_*([\hat\Pi,\hat\Pi])=\Pi$.
\end{dfn}
If $\hat\Pi=\langle \hat X,\hat Y\rangle$ with $\hat X$ and $\hat Y$ defined
by~\eqref{XYhat},~\eqref{XY3} and~\eqref{XY4}, then $\hat\Pi$ is normal if and
only if $[\hat X,\hat Y]$ is a multiple of $\p_\l$. In this case the
integrability condition reduces to the vanishing of the $\p_\l$-component
$\hat X(n)-\hat Y(m)$ of the vector field $[\hat X,\hat Y]$ (identically in
$\l$).

When $d=4$, a generic $2$-plane congruence $\Pi$ has a unique normal lift.
Indeed, generically, $\Delta$ is nonholonomic with $[\Delta,\Delta]=T\hat
M_u$, i.e., it has the growth vector $(3,5)$, and following
Cartan~\cite[\S11]{Cartan1910}, there is a unique rank $2$ subbundle
$\hat{\Pi}\sub\Delta$ with $[\hat{\Pi},\hat{\Pi}]=\Delta$.  Such rank $2$
distribution $\hat{\Pi}$ either has the growth vector $(2,3,5)$ or is
integrable. The former case corresponds to Cartan's celebrated Pfaffian system
\cite{Cartan1910} (for nonintegrable systems or off shell), the latter case
corresponds to a dispersionless Lax pair (on shell).

The genericity condition we need here is as follows (and we formulate a
similar condition for $d=3$ which we will use later).

\begin{dfn} A $2$-plane congruence $\Pi$  is called \emph{nondegenerate} if
\begin{equation}\label{eq:ndg}
\begin{aligned}
\te\wedge\te_\l\wedge \te_{\l\l}&\neq 0,&&\text{where}&
\op{Ann}(\Pi)&=\langle\te\rangle&\text{for}\quad
d&=3;\\ \te\wedge\zeta\wedge\te_\l\wedge\zeta_\l&\neq 0,&&\text{where}&
\op{Ann}(\Pi)&=\langle\te,\zeta\rangle&\text{for}\quad d&=4.
\end{aligned}
\end{equation}
\end{dfn}
These conditions depend only on $\Pi$, not on the choices of $\te$ or
$\zeta$: when $d=4$ nondegeneracy means equivalently
$\varpi_\l\wedge\varpi_\l\neq 0$ where $\varpi=\te\wedge\zeta$, or dually
that $X\wedge Y\wedge X_\l\wedge Y_\l\neq0$, where $\Pi=\langle
X,Y\rangle$. If we choose $\te$ and $\zeta$ as in~\eqref{thet}
and~\eqref{zeth}, then the nondegeneracy conditions may be written explicitly
as:
\begin{align}
\a_\l\b_{\l\l}-\a_{\l\l}\b_\l&\neq 0  &&\text{for}\quad d=3;\label{z1}\\
\a_{\l}\d_{\l}-\b_{\l}\gamma_{\l}&\neq 0 &&\text{for}\quad d=4.\label{z2}
\end{align}

\begin{lem} \label{lem4D}
For $d=4$, any nondegenerate $2$-plane congruence $\Pi$ has a unique normal
lift.
\end{lem}

\begin{proof} If $\hat X$ and $\hat Y$ are given by~\eqref{XYhat}
and~\eqref{XY4}, $\dd x([\hat X,\hat Y])=0 =\dd y([\hat X,\hat Y])$
identically, while $\dd z([\hat X,\hat Y])=\dd t([\hat X,\hat Y])=0$ form two
linear equations on $m,n$:
\begin{equation*}
\begin{bmatrix} \d_\l & -\b_\l \\ -\cg_\l & \a_\l\end{bmatrix}
\begin{bmatrix}m \\ n\end{bmatrix} =
\begin{bmatrix} \a\d_z+\b\d_t-\cg\b_z-\d\b_t+\b_y-\d_x \\
\cg\a_z+\d\a_t-\a\cg_z-\b\cg_t+\cg_x-\a_y\end{bmatrix};
\end{equation*}
these have a unique solution by the nondegeneracy condition~\eqref{z2}.
\end{proof}

\begin{ex}[SDM] We illustrate this with the master equation for SD structures
obtained in~\cite[Theorem 2]{DFK}. Consider a $2$-plane congruence $\Pi$
spanned by~\eqref{XY4} with $\a_\l=0=\d_\l$ and $\b_\l=1=-\cg_\l$. This is
totally isotropic for the conformal class of the metric
\begin{equation*}
g=\te_\l\,\zeta-\zeta_\l\, \te
= \dd x\, (\dd z +\a\,\dd x+\cg\,\dd y)+\dd y\, (\dd t+\b\, \dd x+\d\,\dd y),
\end{equation*}
which is independent of $\l$. In particular, there is a foliation by the
totally isotropic level surfaces of $(x,y)$. Any SD metric can be written in
this form, with the isotropic surface foliation being
anti-self-dual~\cite{DFK,PR}.  The unique normal lift of $\Pi$ is given
by~\eqref{XYhat} with
\begin{align*}
  m= \cg_x -\a_y + \d \a_t -\a \cg_z+\cg \a_z-\b \cg_t,\qquad
  n= \d_x - \b_y + \d \b_t -\a \d_z +\cg \b_z-\b \d_t.
\end{align*}
Now the $\l^2$ term of the integrability condition $\hat X(n)-\hat Y(m)=0$ is
$(\a_z+\cg_t)_z+(\b_z+\d_t)_t=0$, so $\a_z+\cg_t=s_t$ and $\b_z+\d_t=-s_z$ for
some function $s$. However, we may use the translation freedom in $\l$ to set
$s=0$, so that $\a=u_t$, $\cg=-(\l+u_z)$, $\b=\l-v_t$, $\d=v_z$ for
functions $(u,v)$ of $(x,y,z,t)$. Thus we obtain a normal dispersionless pair
$\hat\Pi=\langle\hat X,\hat Y\rangle$ with
\begin{gather*}
\hat X=\p_x-u_t\,\p_z-(\l-v_t)\,\p_t-Q(u)\p_\l,\qquad
\hat Y=\p_y+(\l+u_z)\,\p_z-v_z\,\p_t+Q(v)\p_\l,\\
\text{where}\qquad
Q = \p_x\p_z + \p_y\p_t - u_t \p_z^2 + (u_z+v_t) \p_z\p_t - v_z \p_t^2.
\end{gather*}
The corresponding quasilinear system~\eqref{eqQ} is
\begin{equation}\label{DFK-nLp}
q_x-u_tq_z-(q-v_t)q_t=-Q(u), \qquad q_y+(q+u_z)q_z-v_zq_t=Q(v),
\end{equation}
and the integrability condition reduces to $X(Q(v))+Y(Q(u))=0$, i.e.,
\begin{equation}\label{DFK}
\p_z(Q(u))=\p_t(Q(v)),\qquad
(\p_x-u_t\p_z+v_t\p_t)Q(v)+(\p_y+u_z\p_z-v_z\p_t)Q(u)=0.
\end{equation}
Up to some minor coordinate changes, this is the SD master equation (SDM)
of~\cite{DFK}.
\end{ex}

\subsection{Integrability, dispersionless Lax pairs and normalization}

When $d=3$, we do not obtain a unique normal lift.

\begin{ex}[MS] The dispersionless pair~\eqref{LP-MS} for the Manakov--Santini
system~\eqref{MS} satisfies
\[
[\hat X,\hat Y]=-G \,\p_t - F\,\p_\l
\]
with $F=S(u)+u_t^2$, $G=S(v)$, and so is not normal. However, if we set $\hat
X'=\hat X-G\,\p_\l$ then $\hat X'=\hat X$ on shell (when $F=G=0$), while
\[
[\hat X',\hat Y]=[\hat X,\hat Y]- G\,[\p_\l,\hat Y]+Y(G)\,\p_\l =-(F-G_y
+(\l+v_t)G_t)\,\p_\l
\]
so $\hat\Pi':=\langle \hat X',\hat Y\rangle$ is normal, and is integrable
if and only if
\[
G_t=0,\qquad G_y = F,
\]
i.e., $G=\psi(x,y)$ and $F=\psi_y$. However, this system is not substantively
different from the Manakov--Santini system itself, because we can make a point
transformation $u\mapsto u-\phi_y(x,y)$, $v\mapsto v-\phi(x,y)$ and if
$\phi_{yy}=\psi$, we obtain $F=0$, $G=0$.
\end{ex}

This example illustrates two important issues that we want to incorporate into
the definition of a dispersionless Lax pair $\hat\Pi$ for an equation $\E$:
first $\hat\Pi$ is only determined modulo $\E$, and secondly it can be too
restrictive in examples to require that the integrability conditions for a
dispersionless pair are equivalent to $\E$.

\begin{dfn}\label{DIS} Let $\E:F=0$ be a PDE system on $u$
and $\hat\Pi\sub T\hat M_u$ a dispersionless pair.
\begin{itemize}
\item A dispersionless pair $\hat\Pi'\sub T\hat M_u$ is \emph{$\E$-equivalent}
  to $\hat\Pi$ if $\hat\Pi=\hat\Pi'$ whenever $F(u)=0$.
\item $\hat\Pi$ is a \emph{dispersionless Lax pair} (\emph{dLp}) for $\E$ if
  for any $\hat\Pi'$ $\E$-equivalent to $\hat\Pi$, the integrability condition
  $[\hat\Pi',\hat\Pi']= \hat\Pi'$ is a nontrivial differential corollary of
  $\E$.
\end{itemize}
\end{dfn}

To make precise the notion of a differential corollary, we introduce some jet
formalism, for which we refer to \cite{KV,KL,Vi} for further details. A scalar
PDE of order $\ell$ on a manifold $M$ may be defined as an equation of the
form
\begin{equation}\label{F-ell}
F(j^\ell u)=0
\end{equation}
where $F\in C^\infty(J^\ell M)$ is a function on the bundle $\pi_\ell\colon
J^\ell M\to M$ of $\ell$-jets of functions $u$ on $M$, and $j^\ell u\colon M\to
J^\ell M$ is the $\ell$-jet of $u$, i.e., in coordinates
$j^\ell u=(\bx,u,\p u,\ldots \p^\ell u)$.

In order to discuss objects (such as dLps) depending on an arbitrary finite
jet of $u$, we use the infinite jet bundle $\pi_\infty\colon J^\infty M\to M$
which is the union (inverse limit) of $J^k M$ over all $k$.  A function
$f\colon J^\infty M\to \R$ is smooth if it is the pullback of a function on
$J^k M$ for some $k\in \N$, in which case we say $f$ has \emph{order} $\leq
k$. A choice of coordinates $x^i$ on $M$ leads to coordinates $(x^i,u_\a)$ on
$J^\infty M$, where $1\leq i\leq d$ and $\a$ runs over all symmetric
multi-indices in $d$ entries. Then $f\in C^\infty(J^\infty M)$ has order $\leq
k$ iff it is a function of $x^i$ and $u_\a$ for all $i$ and $\a=(i_1,\ldots i_j)$ 
with $|\a|=j\leq k$.

The bundle $J^\infty M$ has a canonical flat connection, the \emph{Cartan
  distribution}, for which the horizontal lift of a vector field $X$ on $M$
is the \emph{total derivative} $D_X$ characterized by $(D_X f)\circ j^\infty
u = X (f\circ j^\infty u)$ for any smooth function $f$ on $J^\infty M$. More
generally, any section $X$ of $\pi_\infty^* TM$ has a lift to a vector field
$D_X$ on $J^\infty M$, given in local coordinates by $D_X=\sum_i a_iD_i$,
where $X=\sum_i a_i \p_i$ and $D_i=\p_i+\sum_\alpha u_{i\alpha}
\p_{u_\alpha}$.

Higher order operators $\Box$ in total derivatives (also known as
\emph{$\Cc$-differential operators}) are generated as compositions of the
derivations $D_X$ with coefficients being smooth functions on $J^\infty M$.
In local coordinates, $\Box=\sum a_\alpha D_\alpha$, where $a_\alpha\in
C^\infty(J^\infty M)$ and $D_\alpha=D_{i_1}\cdots D_{i_j}$ for a multi-index
$\alpha=(i_1,\ldots i_j)$ with entries in $\{1,2,\ldots d\}$.

Let $\I_F$ be the ideal in $C^\infty(J^\infty M)$ generated by the pullback of
$F\in C^\infty(J^\ell M)$ and its total derivatives of arbitrary order. Then
the zero-set $\E_\infty\sub J^\infty M$ of $\I_F$ is the space of formal solutions
of~\eqref{F-ell}: $u$ is a solution of~\eqref{F-ell} iff $j^\infty u$ is a
section of $\E_\infty$.

These notions extend straightforwardly to PDE systems by replacing $J^\infty
M$ with the bundle $\pi_\infty\colon J^\infty(M,\V)\to M$ of jets of sections
of a fibre bundle $\V\to M$, and $F$ by a function of order $\leq \ell$ on
$J^\infty(M,\V)$ with values in a vector bundle $\W\to M$.
The ideal $\I_F$ in $C^\infty(J^\infty(M,\V))$ is now generated by the
components of $F$ and their total derivatives of arbitrary order.

In this formalism, a \emph{differential corollary} of $\E:F=0$ is a subset
of $\I_F$ (or, more invariantly, the ideal $\I\sub \I_F$ generated by this
subset and its total derivatives of arbitrary order).  It is
\emph{nontrivial} provided it is not a subset of $\I_{F'}$ for any $F'$
whose zero-set in $J^\ell(M,\V)$ contains the zero-set of $F$ in positive
codimension. For example, the ideal generated by $u_{xy}$, for a scalar
function $u(x,y,t)$, is trivial as a differential corollary of the system
$F(j^1u):=(u_x,u_y)=0$, because it is a differential corollary of the
equation $F'(j^1u):=u_xu_y=0$ in which the zero-set of $F$ has positive
codimension. However, it is a nontrivial differential corollary of the
equation $\tilde F(j^1u):=u_x=0$.

Consequently, in Definition~\ref{DIS}, the integrability condition for a dLp
$\hat\Pi$ for $\E:F=0$ need not generate $\I_F$: indeed, the freedom to replace
a dLp by an $\E$-equivalent one may change the ideal $\I\sub \I_F$ that its
integrability conditions generate.

\begin{rk}\label{r:jet} In most of the paper we make minimal use of the jet
formalism by using the philosophy~\cite{KV,Vi} that a differential equation
$\E_\infty\sub J^\infty M$ is a generalized manifold whose ``points'' are
solutions $u$, identified with $M_u=(j^\infty u)(M)\sub \E_\infty$ that is
diffeomorphic to $M$ via $\pi_\infty$. We are justified in working
``pointwise'' provided there are enough ``points'' (i.e., for generic
$u_\infty\in \E_\infty$ there is a solution $u$ with $u_\infty\in M_u$), and there
are existence theorems for hyperbolic PDEs (or rather, ultrahyperbolic PDEs
in signature $(2,2)$) which assert this in some generality. Nevertheless, we
would rather not rely upon such analytical results here, and all our results
can be formalized using jets, even if we do not do so explicitly.
\end{rk}

The following normalization result now suffices to establish
Theorem~\ref{t:proj} when $d=4$.

\begin{prop}\label{p:normal} Let $\hat\Pi$ be a dLp such
that $\Pi=\hat\pi_*(\hat\Pi)$ is nondegenerate. Then $\hat\Pi$ is $\E$-equivalent
to a normal dLp. Such a Lax pair for $d=4$ is unique.
\end{prop}

\begin{proof}
When $d=4$ the Lax pair condition (on shell) implies
\[
\dd z\circ\pi_*[\hat X,\hat Y]=\Box_1F,\qquad \dd t\circ\pi_*[\hat X,\hat Y]
=\Box_2F
\]
for some operators $\Box_1,\Box_2$ in total derivatives.  Let us modify
$\tilde X=\hat X+A(F)\p_\l$, $\tilde Y=\hat Y+B(F)\p_\l$, where $A,B$ are
operators in total derivatives to be determined (they also depend on $\l$).
The new commutation equation is
\begin{gather*}
\dd z\circ\pi_*[\tilde X,\tilde Y]=(\Box_1 + \a_\l B - \cg_\l A)F\\
\dd t\circ\pi_*[\tilde X,\tilde Y]=(\Box_2 + \b_\l B - \d_\l A)F.
\end{gather*}
Vanishing of these, equivalent to normality, can be achieved by a unique
choice of the operators in total derivatives $A,B$ due to nondegeneracy
condition \eqref{z2}.

When $d=3$, the Lax pair condition (on shell) implies similarly
\[
\dd t\circ\pi_*[\hat X,\hat Y]=\hat Y(\a)-\hat X(\b)=\Box F
\]
for some operator $\Box$ in total derivatives.  The modification $\tilde
X=\hat X+A(F)\p_\l$, $\tilde Y=\hat Y+B(F)\p_\l$ gives the new commutation
relations
\[
\dd t\circ\pi_*[\tilde X,\tilde Y]=(\Box+\a_\l B- \b_\l A)F.
\]
The equation $\b_\l A-\a_\l B=\Box$ admits the solution
$A=\a_{\l\l}\Box/(\b_\l\a_{\l\l}-\a_\l \b_{\l\l})$ and
$B=\b_{\l\l}\Box/(\b_\l\a_{\l\l}-\a_\l\b_{\l\l})$ by~\eqref{z1}, unique up to
the freedom $(A,B)\mapsto(A,B)+(\a_\l,\b_\l)L$.
\end{proof}

\section{The characteristic condition for dispersionless Lax pairs}\label{Sec:2}

\subsection{Symbols and the characteristic condition} In order to prove
Theorem~\ref{t:char} in full generality, we need the notions of symbol and
characteristic variety for a general PDE system. For this we use the jet
formalism. Recall from the previous section that a smooth function $F$ on
$J^\infty M$ has order $\leq \ell$ if it is a pullback from $J^\ell M$, and
that $J^\infty M$ has a canonical connection, the Cartan distribution.  The
vertical part of the $1$-form $\dd F\in \Omega^1(J^\infty M)$ may be viewed in
coordinates as a polynomial on $\pi_\infty^* T^*M$ given by
\[
\sum_{j=0}^\ell F_{(j)}\quad\text{where} \quad F_{(j)} =\sum_{|\alpha|=j}
(\p_{u_\alpha} F) \p_\alpha\quad\text{is a section of}\quad\pi_\infty^* S^j TM.
\]
The top degree term $\z_F=F_{(\ell)}$, called the (order $\ell$) \emph{symbol}
of $F$, is independent of coordinates. We assume it is nonvanishing: if it
vanishes, $F$ has order $\leq \ell-1$ and $\z_F$ has lower degree.

This generalizes to a PDE system of order $\ell$, i.e., a function $F$ of
order $\leq \ell$ on $J^\infty(M,\V)$, for some fibre bundle $\V$, with values
in a vector bundle $\W\to M$. The symbol $\z_F$ of $F$ is then a homogeneous
degree $\ell$ polynomial on $\pi_\infty^*T^*M$ with values in
$\op{Hom}(T\V,\W)$, which we assume is not identically zero, so that the PDE
system does not have order $\leq \ell-1$.  The characteristic variety of the
PDE system $\E:F=0$ is defined by~\cite{Sp}
\[
\chv^\E=\{[\te]\in\PP(\pi_\infty^* T^*M)\,|\, \z_F(\te)
\text{ is not injective}\}.
\]
If $\V$ and $\W$ have the same rank, then $[\te]$ is characteristic iff
$\z_F(\te)$ is not surjective. We take $\op{rank}(\V)=\op{rank}(\W)$ as the
definition of a determined system, although a more proper definition is
$\op{codim}\chv^\E=1$.

\begin{dfn} We say that a $2$-plane congruence $\Pi$ (or a dLp $\hat\Pi$)
is \emph{characteristic for $\E$} if for any
solution $u$ of $\E$ and any $\te$ in $\op{Ann}(\Pi)\sub \hat\pi^*T^*M_u$,
we have $[\te]\in \chv^\E$.
\end{dfn}

In the jet formalism, a dispersionless pair $\hat\Pi$ lives on a rank
$1$-bundle $\hat\pi\colon \hat M \to J^\infty(M,\V)$ (so that $\hat
M_u=(j^\infty u)^*\hat M$) and we let $\hat\pi_\infty=\pi_\infty\circ \hat\pi
\colon \hat M\to M$. A $2$-plane congruence $\Pi$ is then a rank $2$ subbundle
of $\hat\pi_\infty^* TM$, and $\hat\Pi$ is a lift of $\Pi$ to $T\hat M$. In
practice we use a spectral parameter $\l$ to trivialize $\hat M$ over
$J^\infty(M,\V)$. Then $T\hat M$ is the direct sum of the vertical bundle of
$\hat\pi$, spanned by $\p_\l$, and $\hat\pi^* TJ^\infty(M,\V)$. Thus if $\Pi$
is spanned by $X,Y\in \hat\pi_\infty^* TM$, we may write the dispersionless
pair $\hat\Pi$ as the span of $\hat X = D_X + m\,\p_\l$ and $\hat Y= D_Y+
n\,\p_\l$, where $D_X$ and $D_Y$ are total derivatives (depending also on
$\l$) and $m,n$ are functions on $\hat M$.  Then
\[
[\hat X,\hat Y] = \bigl([D_X,D_Y] + m\,D_{\p_\l Y} - n\,D_{\p_\l X}\bigr)
+ \bigl( D_X n - D_Y m + m\,\p_\l n - n\,\p_\l m\bigr)\,\p_\l.
\]
The integrability condition $[\hat X,\hat Y]\in\Gamma(\hat\Pi)$ reduces to
$[X,Y] + m\,\p_\l Y -n\,\p_\l X=\nu_X\, X+\nu_Y\, Y$, for some $\nu_X,\nu_Y$,
together with the vanishing of $D_X n - D_Y m + m\,\p_\l n - n\,\p_\l
m-\nu_X\, m-\nu_Y\, n$. As in the previous section, we may choose $X$ and $Y$
so that $\nu_X=\nu_Y=0$, and hence the Lax equation (split into the vertical
and horizontal parts) becomes the system
\begin{align}\label{main}
D_X n - D_Y m + m\, \p_\l n - n\, \p_\l m = 0,\\
\label{second}
D_X Y - D_Y X + m\, \p_\l Y - n \,\p_\l X=0.
\end{align}
We thus have a dLp for $\E$ if these equations hold modulo $\I_F$ i.e., all
components (and hence their total derivatives of arbitrary order) belong to
$\I_F$.

\begin{lem}\label{key-lemma1} If $D_X q - D_Y p$ has order $\leq k$, for
functions $p,q$ of $u_\infty\in J^\infty (M,\V)$ and sections $X,Y$ of
$\pi_\infty^* TM$, then its order $k$ symbol is
\begin{equation}\label{eq:symbol}
D_{X_{(k)}} q + D_X (q_{(k)}) +X\odot q_{(k-1)}
- D_{Y_{(k)}} p  - D_Y (p_{(k)}) -Y\odot p_{(k-1)}.
\end{equation}
If $X,Y$ are linearly independent, and $P_1$ and $P_2$ are symmetric
$k$-vectors with $X\odot P_2=Y\odot P_1$, there is a symmetric $(k-1)$-vector
$S$ with $P_1=X\odot S$ and $P_2=Y\odot S$.
\end{lem}

\begin{proof} Equation~\eqref{eq:symbol} is straightforward from the
definition of the total derivative and the product rule for the vertical
differentiation. Extending $X,Y$ pointwise to a basis, the second part reduces
to the trivial observation that for any homogeneous polynomials
$P_j=P_j(\xi_1,\ldots \xi_d)$, $j=1,2$, with $\xi_1 P_2=\xi_2 P_1$, there is a
homogeneous polynomial $P$ with $P_j=\xi_j P$.
\end{proof}

\begin{lem}\label{key-lemma2}
Let~\eqref{main}--\eqref{second} have
order $\leq k+1$ modulo $\I_F$, i.e., all their higher symbols vanish modulo
$\I_F$. Then there is a symmetric $k$-tensor $S_k$ and a symmetric
$TM$-valued $k$-tensor $Q_k$ such that, modulo $\I_F$, the order $k+1$ symbols
of~\eqref{main} and \eqref{second} are respectively
\begin{gather}
\label{symbol-formula}
\begin{multlined}[c][.8\textwidth]
X\odot (n_{(k)}+D_{Q_k}n-D_Y S_k+S_k\,\p_\l n- n \,\p_\l S_k)\\
-Y\odot (m_{(k)}+D_{Q_k}m-D_X S_k+S_k\,\p_\l m- m \,\p_\l S_k),
\end{multlined}\\
\label{second-formula}
\begin{multlined}[c][.8\textwidth]
X\odot (Y_{(k)}+D_{Q_k}Y-D_Y Q_k+S_k\,\p_\l Y-n\,\p_\l Q_k)\\
- Y\odot (X_{(k)}+D_{Q_k}X-D_X Q_k+S_k\,\p_\l X-m\,\p_\l Q_k).
\end{multlined}
\end{gather}
\end{lem}

\begin{proof} Suppose that $X,Y,m,n$ depend only on the $N$-jet of $u$ for
some $N\in\N$, so that~\eqref{main}--\eqref{second} have order $\leq N+1$, and
it suffices to prove the lemma for $k\leq N$. We thus induct on $p=N-k$.  For
$p=0$, the order $k+1=N+1$ symbols of~\eqref{main} and~\eqref{second} are
simply $X\odot n_{(k)} - Y\odot m_{(k)}$ and $X\odot Y_{(k)} - Y\odot X_{(k)}$
by~\eqref{eq:symbol}, so we are done, with $S_k=0=Q_k$.

Now suppose that the lemma holds with $k=N-p$ for some $p\geq0$, and suppose
that~\eqref{main}--\eqref{second} have order $\leq k$ modulo $\I_F$.
Then~\eqref{main} certainly has order $\leq k+1$ modulo $\I_F$, and so the
inductive hypothesis implies its order $k+1$ symbol, which vanishes modulo
$\I_F$, is given by~\eqref{symbol-formula}. Hence Lemma~\ref{key-lemma1}
produces a symmetric $(k-1)$-tensor $S_{k-1}$ such that, modulo $\I_F$,
\begin{align*}
m_{(k)}&= X\odot S_{k-1} - D_{Q_k}m + D_X S_k - S_k\,\p_\l m + m \,\p_\l S_k,\\
n_{(k)}&= Y\odot S_{k-1} - D_{Q_k}n + D_Y S_k - S_k\,\p_\l n + n \,\p_\l S_k.
\end{align*}
Similarly, by~\eqref{second-formula}, there is a symmetric $TM$-valued
$(k-1)$-tensor $Q_{k-1}$ such that
\begin{align*}
X_{(k)}&= X\odot Q_{k-1} - D_{Q_k}X + D_X Q_k - S_k\,\p_\l X + m\,\p_\l Q_k,\\
Y_{(k)}&= Y\odot Q_{k-1} - D_{Q_k}Y + D_Y Q_k - S_k\,\p_\l Y + n\,\p_\l Q_k,
\end{align*}
modulo $\I_F$. By~\eqref{eq:symbol}, the order $k$ symbol of~\eqref{main} is
\begin{multline*}
D_{X_{(k)}} n + D_X (n_{(k)}) + X\odot n_{(k-1)}
- D_{Y_{(k)}} m - D_Y (m_{(k)}) - Y\odot m_{(k-1)}\\
+ m_{(k)}\,\p_\l n + m \,\p_\l (n_{(k)})
- n_{(k)}\,\p_\l m - n \,\p_\l (m_{(k)}).
\end{multline*}
Hence, substituting for $X_{(k)}, Y_{(k)}, m_{(k)}, n_{(k)}$, we have
\begin{multline*}
  D^{\vphantom{g}}_{X\odot Q_{k-1} - D_{Q_k}X + D_X Q_k - S_k\,\p_\l X + m\,\p_\l Q_k} n
- D^{\vphantom{g}}_{Y\odot Q_{k-1} - D_{Q_k}Y + D_Y Q_k - S_k\,\p_\l Y + n\,\p_\l Q_k} m\\
- D_Y (X\odot S_{k-1} - D_{Q_k}m + D_X S_k - S_k\,\p_\l m + m \,\p_\l S_k)\\
+ D_X (Y\odot S_{k-1} - D_{Q_k}n + D_Y S_k - S_k\,\p_\l n + n \,\p_\l S_k)\\
+ (X\odot S_{k-1} - D_{Q_k}m + D_X S_k - S_k\,\p_\l m + m \,\p_\l S_k)\,\p_\l n\\
- (Y\odot S_{k-1} - D_{Q_k}n + D_Y S_k - S_k\,\p_\l n + n \,\p_\l S_k)\,\p_\l m\\
- n \,\p_\l (X\odot S_{k-1} - D_{Q_k}m + D_X S_k - S_k\,\p_\l m + m \,\p_\l S_k)
+ \rlap{$X\odot n_{(k-1)}$}\\
+ m \,\p_\l (Y\odot S_{k-1} - D_{Q_k}n + D_Y S_k - S_k\,\p_\l n + n \,\p_\l S_k)
- Y\odot m_{(k-1)}.
\end{multline*}
A lot of cancellation now occurs to leave
\begin{multline*}
  X\odot (n_{(k-1)} + D_{Q_{k-1}}n - D_YS_{k-1} + S_{k-1}\,\p_\l n - n\,\p_\l S_{k-1})\\
- Y\odot (m_{(k-1)} + D_{Q_{k-1}}m - D_XS_{k-1} + S_{k-1}\,\p_\l m - m\,\p_\l S_{k-1})\\
+ (D_XY -D_YX +m \,\p_\l Y-n\,\p_\l X) \odot S_{k-1}
+ (D_Xn - D_Ym + m\,\p_\l n - n\,\p_\l m)\,\p_\l S_k\\
+ D_{D_XY - D_Y X + m \,\p_\l Y - n \,\p_\l X} S_k
- (D_{Q_k} + S_k\,\p_\l) (D_X n - D_Y m + m\, \p_\l n - n\, \p_\l m)
\end{multline*}
and the last two lines vanish modulo $\I_F$, which
establishes~\eqref{symbol-formula} for $k'=N-(p+1)=k-1$.  We turn now to the
order $k$ symbol of~\eqref{second}, which, by~\eqref{eq:symbol}, is
\begin{multline*}
D_{X_{(k)}} Y + D_X(Y_{(k)}) + X\odot Y_{(k-1)}
- D_{Y_{(k)}} X - D_Y(X_{(k)}) - Y\odot X_{(k-1)}\\
+ m_{(k)} \, \p_\l Y + m\,\p_\l Y_{(k)}
- n_{(k)} \,\p_\l X - n\,\p_\l X_{(k)}.
\end{multline*}
Hence, substituting for $X_{(k)}, Y_{(k)}, m_{(k)}, n_{(k)}$, we have, modulo $\I_F$,
\begin{align*}
0&= D_{X\odot Q_{k-1} - D_{Q_k}X + D_X Q_k - S_k\,\p_\l X + m\,\p_\l Q_k} Y
- D_{Y\odot Q_{k-1} - D_{Q_k}Y + D_Y Q_k - S_k\,\p_\l Y + n\,\p_\l Q_k} X\\
&\qquad- D_Y(X\odot Q_{k-1} - D_{Q_k}X + D_X Q_k - S_k\,\p_\l X + m\,\p_\l Q_k)\\
&\qquad+ D_X(Y\odot Q_{k-1} - D_{Q_k}Y + D_Y Q_k - S_k\,\p_\l Y + n\,\p_\l Q_k)\\
&\qquad
+ (X\odot S_{k-1} - D_{Q_k}m + D_X S_k - S_k\,\p_\l m + m \,\p_\l S_k) \, \p_\l Y\\
&\qquad
- (Y\odot S_{k-1} - D_{Q_k}n + D_Y S_k - S_k\,\p_\l n + n \,\p_\l S_k ) \,\p_\l X\\
&\qquad
- n\,\p_\l ( X\odot Q_{k-1} - D_{Q_k}X + D_X Q_k - S_k\,\p_\l X + m\,\p_\l Q_k )
+ \rlap{$X\odot Y_{(k-1)}$}\\
&\qquad
+ m\,\p_\l ( Y\odot Q_{k-1} - D_{Q_k}Y + D_Y Q_k - S_k\,\p_\l Y + n\,\p_\l Q_k )
- Y\odot X_{(k-1)}\displaybreak[0]\\
&= X\odot(Y_{(k-1)} + D_{Q_{k-1}}Y - D_YQ_{k-1} + S_{k-1}\,\p_\l Y- n\,\p_\l  Q_{k-1})\\
&\qquad
- Y\odot(X_{(k-1)} + D_{Q_{k-1}} X - D_XQ_{k-1}+ S_{k-1}\,\p_\l X - m\,\p_\l Q_{k-1})\\
&\qquad+( D_X Y - D_Y X+ m\,\p_\l Y- n\,\p_\l X)\odot Q_{k-1}\\
&\qquad+ (D_Xn-D_Ym+ m\,\p_\l n- n\,\p_\l m)\,\p_\l Q_k,\\
&\qquad+ D_{D_XY-D_YX+m\,\p_\l Y-n\,\p_\l X} Q_k
- (D_{Q_k}+S_k\,\p_\l)(D_XY-D_YX+ m \, \p_\l Y -n\,\p_\l X)
\end{align*}
and the last two lines again vanish modulo $\I_F$, so
that~\eqref{second-formula} holds for $k'=N-(p+1)=k-1$, completing the proof.
\end{proof}

\subsection{Proof of Theorem~\textup{\ref{t:char}}} The strategy is to find
a dispersionless pair $\E$-equivalent to $\hat\Pi$ whose integrability
condition has minimal order.

We may assume as above that $\hat\Pi$ is spanned by vector fields
$D_X+m\,\p_\l$ and $D_Y+n\,\p_\l$ which commute on shell, where $X,Y,m,n$
depend only on the $N$-jet of $u$ for some $N\in\N$, and
that~\eqref{main}--\eqref{second} have order $\leq k+1$, modulo the ideal
$\I_F$ generated by $F$ and its total derivatives, for $\ell-1\leq k\leq N$.
By the definition of a dLp, these equations have the form $\Lambda_1(F)=0$
and $\Lambda_2(F)=0$, where $\Lambda_1$ and $\Lambda_2$ are $\l$-dependent
operators in total derivatives, the latter being $TM$-valued. In local
coordinates we may write $\Lambda_1$ as a finite sum $\sum_\alpha
b_\alpha(u_\infty,\l) D_\alpha$, and then the symbol of $\Lambda_1(F)$ of
any order $r\geq \ell+1$ is
\[
\Lambda_1(F)_{(r)}=\sum_{j=0}^\ell \sum_{|\alpha|= r - j} b_\alpha(u_\infty,\l)
\p_\alpha \odot F_{(j)}\mod \I_F.
\]
Since the order $r$ symbol vanishes modulo $\I_F$ for $r\geq k+2$, we deduce,
starting from $r=\max\{|\alpha|:b_\alpha\neq 0\}+\ell$, that
$b_\alpha=0\mod\I_F$ for $|\alpha|\geq k-\ell+2$ and that, for $k\geq \ell$, the
order $k+1$ symbol has the form $L_1\odot \z_F$ modulo $\I_F$, where
$\z_F=F_{(\ell)}$ and $L_1$ is a symmetric $(k-\ell+1)$-vector depending on
$(u_\infty,\l)$; this also holds straightforwardly when $k=\ell-1$.
Similarly, for any $k\geq \ell-1$, the order $k+1$ symbol of $\Lambda_2(F)$
has the form $L_2\odot \z_F$ modulo $\I_F$ for a $TM$-valued symmetric
$(k-\ell+1)$-vector depending on $(u_\infty,\l)$.

By Lemma~\ref{key-lemma2}, these symbols have the
form~\eqref{symbol-formula}--\eqref{second-formula} modulo $\I_F$. Hence, on
any solution $u$ and for any $\te\in \op{Ann}(\Pi)$, we have
$L_1(\te)\circ\z_F(\te)=0$ and $L_2(\te)\circ\z_F(\te)=0$ (there
is only one independent $\te$ at each point for $d=3$ and a pair for
$d=4$).

For $k=\ell-1$, $(L_1,L_2)$ is a nonzero $(d+1)$-vector-valued function of
$(u_\infty,\l)$. Hence $\z_F(\te)=0$ for all $\te\in \op{Ann}(\Pi)$
and we are done. We may thus induct on $k\geq\ell-1$, and suppose that the
result holds when~\eqref{main}--\eqref{second} have order $\leq k$.  We either
have $L_1(\te)=0$ and $L_2(\te)=0$ as polynomials in
$\te\in\op{Ann}(\Pi)$, or that $\z_F(\te)$ is not surjective for all
such $\te$. In the latter case, we are done, since the PDE is
determined. The former case implies that $L_1=X\odot T_1-Y\odot U_1$ and
$L_2=X\odot T_2-Y\odot U_2$ for some symmetric $(k-\ell)$-vectors
$T_1,U_1,T_2,U_2$ (the latter pair being $TM$-valued).

We now let $\tau_1,\upsilon_1,\tau_2,\upsilon_2$ be order $k-\ell$ operators
in total derivatives such that $\tau_1 F$ has order $k$ symbol $T_1\odot\z_F$
modulo $\I_F$ and so on: concretely, in local coordinates, if
$T_1=\sum_{|\alpha|=k-\ell} t_\alpha(u_\infty,\l) \p_\alpha$, we may take
$\tau_1 = \sum_{|\alpha|=k-\ell} t_\alpha(u_\infty,\l) D_\alpha$. We then
modify the dispersionless pair by $m\mapsto m-\upsilon_1 (F)$, $n\mapsto
n-\tau_1 (F)$, $X\mapsto X-\upsilon_2 (F)$, $Y\mapsto Y-\tau_2 (F)$. This
modification is $\E$-equivalent to $\hat\Pi$, but the new order $k+1$ symbols
of~\eqref{main}--\eqref{second} vanish modulo $\I_F$, so they have order $\leq
k$ modulo $\I_F$, and the result follows by the inductive hypothesis.
\qed

\subsection{Dispersionless pairs characteristic for a quadric}\label{S:Q}

If $\hat\Pi$ is a dLp for an equation $\E$ whose characteristic variety
$\chv^\E$ is a quadric, then $\Pi$ is coisotropic for this quadric by
Theorem~\ref{t:char}. In this section we investigate the extent to which $\Pi$
recovers this quadric. We begin with a uniqueness criterion, and then discuss
existence. We make essential use of the nondegeneracy
conditions~\eqref{z1}--\eqref{z2}, which imply in particular that at each
$\bx\in M_u$, the image of $\Pi_\bx\colon\l\mapsto\Pi_{(\bx,\l)}$ does not lie
in any proper projective linear subspace of $\op{Gr}_2(T_\bx M_u)
\sub\PP(\Wedge^2T_\bx M_u)$.

\begin{prop}\label{prp-u}
If a $2$-plane congruence $\Pi$ is coisotropic for $c_F$, then for any $\bx\in
M_u$ and $\l\in\hat\pi^{-1}(\bx)$ at which $\Pi_\bx$ is an immersion, it is
nondegenerate at $\bx$. Conversely, at any point $\bx$ where $\Pi$ is
nondegenerate, there is at most one \textup(quadratic\textup) conformal
structure $c_F$ on $T_{\bx} M_u$ with $\Pi_{(\bx,\l)}$ coisotropic for all
$\l$, and it must be nondegenerate and hyperbolic.
\end{prop}

\begin{proof}
Suppose first that $d=3$, so that $\op{Gr}_2(T_\bx M_u)\cong
\PP(T^*_{\bx}M_u)$ is a projective plane, and $\op{Ann}(\Pi_{\bx})$ is a curve
in this plane. If $\Pi$ is coisotropic, then $\op{Ann}(\Pi_{\bx})$ lies on the
nonsingular conic $\{[\te]:\z_F(\te)=0\}$ and so if $\Pi_{\bx}$ is immersed,
its derivatives of order $\leq 2$ in $\l$ span $\PP(T^*_{\bx}M_u)$, hence it
is nondegenerate. Conversely, two distinct nonsingular conics meet in at most
four points, so $\op{Ann}(\Pi_{\bx})$ lies on at most one nonsingular conic
(which is nonempty, hence hyperbolic), and if $\op{Ann}(\Pi_{\bx})$ lies on a
singular conic, it lies on a line, hence $\Pi_{\bx}$ is degenerate.

Suppose instead that $d=4$, so that (the Pl\"ucker embedding of)
$\op{Gr}_2(T_\bx M_u)$ is the Klein quadric in $\PP(\Wedge^2T_{\bx}M_u)$, and
$\Wedge^2\Pi_{\bx}$ is a curve in this quadric.  If $\Pi$ is coisotropic, then
$\Wedge^2\Pi_{\bx}$ lies in a nondegenerate plane section of this quadric,
which is a conic: the corresponding lines in $\PP(T^*_{\bx}M_u)$ belong to one
of the rulings of the quadric surface $\{[\te]:\z_F(\te)=0\}$ in
$\PP(T^*_{\bx}M_u)$. In particular, if $\Pi_{\bx}$ is immersed, its tangent
does not lie in the quadric, hence it is nondegenerate.  Conversely, two
distinct nonsingular quadric surfaces meet in a degree four curve (containing
at most four lines), so if $\Pi_\bx$ is nonconstant, it lies on at most one
nonsingular quadric surface (which is hyperbolic because it contains lines),
and if $\Pi_\bx$ has image in a singular quadric surface, then the lines pass
through a point or lie in a plane, hence $\Wedge^2\Pi_\bx$ lies in a proper
projective linear subspace of $\op{Gr}_2(T_\bx M_u)$, hence $\Pi_{\bx}$ is
degenerate.
\end{proof}

\begin{prop}\label{prp-e}
Suppose $d=3$ and the nondegeneracy condition~\eqref{z1} holds. Then there is
a unique conformal structure $c$ for which the $2$-plane congruence
$\Pi=\langle X,Y\rangle$ is null for all $\l$ if and only if the Monge
invariant $I(\a,\b)=0$. This invariant has order $5$ in the entries and it
distinguishes conics in the projective plane. In the local parametrization
with $\b=\l$, this condition is the following \textup(we denote $\a'=\a_\l$
etc.\textup{):}
\begin{equation*}
I(\a,\l)=9(\a'')^2\a^{(5)}-45\a''\a'''\a^{(4)}+40(\a''')^2=0.
\end{equation*}

Suppose $d=4$ and the nondegeneracy condition~\eqref{z2} holds. Then there is
a unique conformal structure $c$ for which the $2$-plane congruence
$\Pi=\langle X,Y\rangle$ is \textup(co\textup)isotropic for all $\l$ if and
only if the following system of differential equations of order $3$ holds,
which we write in a partially integrated second order form so \textup(again
$\a'=\a_\l$ etc.\textup)
\begin{gather*}
v'w''-v''w'=k_{vw}|\a'\d'-\b'\cg'|^{3/2}\ \text{ for }v,w\in\{\a,\b,\cg,\d\},
\end{gather*}
where $k_{vw}$ are $\l$-independent and satisfy the ``cocycle conditions''
$k_{vw}+k_{wv}=0$, $u'k_{vw}+v'k_{wu}+w'k_{uv}=0$ for
$u,v,w\in\{\a,\b,\cg,\d\}$.  In the normalization $\d=\l$ these conditions
simplify to: $(\a,\b,\cg)''=\mathbf{v}\, |\a'-\b'\cg'|^{3/2}$, where
$\mathbf{v}$ is a $\l$-independent $3$-component vector.
\end{prop}

\begin{proof}
Let us discuss first the case $d=3$. We are looking for a conformal structure
$c$, represented by a pseudo-Riemannian metric $g$ of signature $(2,1)$, such
that the planes $\Pi=\langle X,Y\rangle$ are null.  Consider the Pfaffian form
$\te=\dd t+\a\,\dd x+\b\,\dd y\in\op{Ann}(\Pi)$.  The null condition is a
single equation $c(\te,\te)=0$. Adding to it its $\l$-derivatives up to
order 4, we get a system of 5 equations on 6 coefficients of the metric (5
coefficients if considered up to proportionality).  This system is solvable
iff (\ref{z1}) holds. Provided this nondegeneracy condition, we can uniquely
find $c=[g]$, but in order for it to be supported on $M_u$ (and not on $\hat
M_u$) the ratio of the coefficients of $g$ must be $\l$-independent. This is
equivalent to the condition $I(\a,\b)=0$.

Consider now the case $d=4$. Add to the 3 equations $c(X,X)=0$, $c(X,Y)=0$,
$c(Y,Y)=0$ their first and second derivatives in $\l$.  The obtained system of
9 equations on 10 coefficients of the metric (9 coefficients if considered up
to proportionality) is solvable iff condition (\ref{z2}) holds.  Provided this
nondegeneracy condition, we can uniquely find $c=[g]$, but in order for it to
be supported on $M_u$ (and not on $\hat M_u$) the ratio of the coefficients of
$g$ must be $\l$-independent. This is equivalent to the system of equations
formulated in the proposition.
\end{proof}

\section{Projective dependence on the spectral parameter}\label{Sec:3}

\subsection{Weyl connections and standard dLps}\label{S:N}

For any $2$-plane congruence $\Pi$ which is characteristic for a bundle of
nonsingular quadric hypersurfaces, there is a well-known construction of lifts
$\hat\Pi_\nabla$ of $\Pi$ from \emph{Weyl connections} $\nabla$, i.e., a
torsion-free connections preserving the conformal structure $c$ defining the
quadric. Such Weyl connections form an affine space modelled on the vector
space of $1$-forms on $M_u$.

\begin{lem}\label{l:weyl} Let $\Pi$ be a nondegenerate $2$-plane congruence
on $\hat M_u\to M_u$, characteristic for a bundle of quadric
hypersurfaces, and $\nabla$ a Weyl connection. Then $\nabla$ induces a
connection on $\hat M_u$ such that the horizontal lift
$\hat\Pi_\nabla$ of $\Pi$ is normal.
\end{lem}
\begin{proof} Since $\nabla$ is a conformal connection, it induces a
connection on the bundle of coisotropic planes for $c$, and hence on $\hat
M_u$, since $\Pi$ is an immersion. The pullback of $\nabla$ to $\hat\pi^*
TM_u$ preserves $\Pi$ and hence, since $\nabla$ is torsion-free, the
horizontal lift $\hat\Pi$ satisfies $\hat\pi_*[\hat\Pi,\hat\Pi]=\Pi$.
\end{proof}

We refer to such a lift $\hat\Pi_\nabla$ as a \emph{standard dLp}. For $d=4$,
any standard dLp $\hat\Pi_\nabla$ is the unique normal lift of
$\Pi=\hat\pi_*(\hat\Pi_\nabla)$, hence independent of the choice of Weyl
connection, as is well known~\cite{Pe}. However, for both $d=3$ and $d=4$,
standard dLps are very special because the connection induced by $\nabla$ on
$\hat M_u$ is \emph{projective}: $\hat M_u$ is locally isomorphic to a
$\PP^1$-bundle over $M_u$ and if $\l$ is a spectral parameter induced by an
affine coordinate on this projective bundle, then horizontal lifts of
($\l$-independent) vector fields on $M_u$ depend quadratically on $\l$
(because vector fields on $\PP^1$ have this form in an affine chart).

Furthermore, with respect to such a projective spectral parameter $\l$, there
is a local parametrization of vector fields spanning $\Pi$ that is linear in
$\l$, i.e., $\Pi=\langle V_1+\l V_3,V_2+\l V_4\rangle$ for $\l$-independent
vector fields $V_i$ on $M_u$, so that their lifts are cubic in $\l$.

When $d=4$, these properties follow from the existence of an adapted
frame $V_1,V_2,V_3,V_4$ for $M_u$ such that in the dual coframe
$\te_1,\te_2,\te_3,\te_4$, the conformal structure is represented by
$g=\te_1\te_4-\te_2\te_3$. Then (up to a choice of orientation)
$\Pi=\langle V_1+\l V_3,V_2+\l V_4\rangle$, $\Delta=\langle V_1+\l
V_3,V_2+\l V_4,\p_\l\rangle$ and it is straightforward to verify that
the unique normal lift of $\Pi$ is
\[
\hat{\Pi}=\langle V_1+\l V_3+m\p_\l,V_2+\l V_4+n\p_\l\rangle,
\]
where the coefficients $m,n$ are given in terms of the structure functions
$c_{ij}^k=\te_k([V_i,V_j])$ of the frame as
\begin{align*}
m&=-c_{12}^4+\l(c_{23}^4-c_{14}^4+c_{12}^2)-\l^2(c_{23}^2-c_{14}^2+c_{34}^4)
+\l^3c_{34}^2, \\
n&=c_{12}^3-\l(c_{23}^3-c_{14}^3+c_{12}^1)+\l^2(c_{23}^1-c_{14}^1+c_{34}^3)-\l^3c_{34}^1.
\end{align*}
These are cubic in $\l$ as required, and compatible with the
representation $m=m_1+\l m_3$ and $n=m_2+\l m_4$ for coefficients
$m_i$ of $\p_\l$ in the lifts of $V_i$ that are quadratic in $\l$.

When $d=3$, there is similarly an adapted frame $V_0,V_1,V_2$ on $M_u$ with
the dual coframe $\te_0,\te_1,\te_2$ such that conformal structure $c_F$ is
represented by the Lorentzian metric $g=4\te_0\te_2-\te_1^2$ and $\Pi=\langle
V_0+\l V_1,V_1+\l V_2\rangle=\ker\te(\l)$, where
\begin{equation}\label{eq:thl}
\te(\l)=\te_2-\l\te_1+\l^2\te_0
\end{equation}
for a (projective) spectral parameter $\l$. We then have the following
fact (cf.~\cite{DMT}).

\begin{lem}\label{lem3D}
Let $d=3$ and let $\Pi$ be as in Lemma~\textup{\ref{l:weyl}}. Then Weyl
connections parametrize projective normal lifts $\hat\Pi$ of $\Pi$.
\end{lem}
\begin{proof} By definition any projective lift given by~\eqref{XYhat},
with $X=V_0+\l V_1,Y=V_1+\l V_2$ affine linear, has $m,n$ cubic in $\l$, i.e.,
$m=\sum_{i=0}^3m_i\l^i$, $n=\sum_{i=0}^3n_i\l^i$.  Now $\op{Ann}(\Pi)$ is
spanned by the $1$-form~\eqref{eq:thl} where $\te_i(V_j)=\d_{ij}$. Hence
$\te(\hat\pi_*[\hat X,\hat Y])= \te([X,Y]-nV_1+mV_2)$ is a quartic
polynomial in $\l$ determining $5$ of the $8$ coefficients of $m$ and $n$.  It
is straightforward to check that remaining three coefficients are determined
uniquely by the Weyl connection (a $1$-form has three components at each
point).
\end{proof}

\subsection{The modified Manakov--Santini master equation in 3D}

As mentioned in the introduction, the integrability for a standard dLp
$\hat\Pi_\nabla$ in 3D is well-known to be equivalent to the EW equation on
$(c,\nabla)$ and has the geometric interpretation that any EW manifold locally
admits (many) foliations by totally geodesic null surfaces~\cite{Cartan,Hi}
(corresponding to curves in the minitwistor space). We now use this to obtain
an alternative derivation of the Manakov--Santini system~\cite{MS} as a master
equation in 3D, or rather a modification of this system which was previously
derived in~\cite{DFK} by a different method.

Any totally geodesic null surface has a canonical foliation by null geodesics,
so any EW manifold admits a local coordinate system $(x,y,t)$, where $x$ and
$y$ are pulled back from local coordinates on the local leaf spaces of a
totally geodesic null surface foliation and the induced null geodesic
foliation respectively. Thus $\p_t$ is null and orthogonal to $\p_y$ and we
can use the freedom in the $t$ coordinate so that the conformal structure has
a representative metric
\begin{equation}\label{eq:EWg}
g= 4 (\dd t - b\, \dd x) \dd x - (\dd y - a\, \dd x)^2
\end{equation}
for some functions $a$ and $b$. This has the form $4\te_0\te_2-\te_1^2$, where
$\te_0=\dd x$, $\te_1=-\dd y + a\, \dd x$, $\te_2=\dd t - b\, \dd x$ is the
coframe dual to $V_0=\p_x + a \p_y + b \p_t$, $V_1=-\p_y$ and $V_2=\p_t$.
Thus the null $2$-plane congruence $\Pi=\langle V_0+\l V_1, V_1+\l V_2\rangle$
is the kernel of
\begin{equation}\label{thla}
\te(\l) = (\dd t - b\,\dd x) + \l (\dd y - a\, \dd x) + \l^2 \dd x
= \dd t + \l \,\dd y + (\l^2 - a\l - b) \dd x,
\end{equation}
and is equal to $W_0^\perp$ where
\begin{equation}\label{eq:W0MS}
W_0=V_0+2\l V_1+ \l^2 V_2 = \p_x + a \p_y + b \p_t -2 \l \p_y + \l^2 \p_t
=\p_x+(a-2\l)\p_y+(b+\l^2)\p_t.
\end{equation}
Since $\p_y$ and $\p_t$ are tangent to level surfaces of $x$, which are the
null surfaces corresponding to $\l=\infty$, the standard dLp must have the
form $\hat\Pi_\nabla=\langle V_0+\l V_1+m'\p_\l, V_1+\l V_2+n'\p_\l\rangle $
where $m'$ and $n'$ are quadratic in $\l$.

To obtain a $2$-plane congruence in the form~\eqref{XY3}, we let 
$X= V_0+\l V_1 + (a-\l)(V_1+\l V_2)$ and $Y=-(V_1+\l V_2)$ so that
$\hat X$ and $\hat Y$ are given by~\eqref{XYhat} with $m=m' + (a-\l)n'$ and
$n=-n'$. The Lax integrability condition $[\hat X,\hat Y]=0$ implies $n'$ is
affine linear in $\l$, while $m'$ is a quadratic in $\l$, where the
coefficient $h$ of $\l^2$ is a function of $x$ and $y$. We may set $h$ to zero
using the coordinate freedom
\[
x\mapsto x,\quad y\mapsto \rho(x,y),\quad t\mapsto \rho_y(x,y)^2t,\quad
\l\mapsto \rho_y(x,y)(\l-2\rho_{yy}(x,y)t),
\]
which preserves the form of $\te(\l)$ (hence also $g$) up to rescaling by
$\rho_y(x,y)^2$ and a redefinition of $a$ and $b$. The Lax equation now
implies that the $\l$ coefficient of $m'$ differs from $-a_y$ by a function of
$x$ and $y$ which may be set to zero using the remaining coordinate freedom
\[
x\mapsto x,\quad y\mapsto y,\quad t\mapsto t+\tau(x,y),\quad \l\mapsto
\l-\tau_y(x,y).
\]
We then find that $m'=-a_y\l -b_y$, $n'=a_t \l + b_t$, and hence
\begin{align*}
\hat X &=
\p_x + (-\l^2 + a \l+ b) \p_t - ((a_y\l + b_y)+(\l-a)(a_t \l+b_t))\p_\l,\\
\hat Y &= \p_y - \l \p_t - (a_t \l+b_t)\p_\l.
\end{align*}
The Lax integrability condition now reduces to the determined system
\begin{equation}\label{eq:EW}
  (a_x-aa_y +ba_t)_t = (a_y -2 a a_t)_y,\qquad
  (b_x-ab_y +bb_t)_t = (b_y -2 a b_t)_y.
\end{equation}
This is the form of the EW system given in~\cite[(11)--(12)]{DFK}, except that
the $x$ and $t$ variables have been swapped in our conventions and we have
used the identity $(aa_y)_t=(aa_t)_y$. Substituting $a=v_t$ and $b=u-v_y$
gives the Manakov--Santini system. The modified version~\eqref{eq:EW} may
also be written more geometrically as
\begin{equation*}
\Delta^g a =0, \qquad \Delta^g b +\tfrac3 2\{a,b\}_P = 0,
\end{equation*}
where $\Delta^g$ is the Laplacian of the metric $g$ in~\eqref{eq:EWg}, and
$\{a,f\}_P = a_y f_t - a_t f_y$ is the Poisson bracket with respect to the
bivector field $P=\p_y\wedge\p_t$ tangent to the null surface foliation.

\begin{rk} In \cite{DFK} a translationally noninvariant version of the MS
system was also derived and the question of an explicit equivalence to the
standard MS system was raised. However, the translationally noninvariant
version is obtained from a generic null surface foliation of the EW manifold,
and the coordinate transformation to a totally geodesic null surface foliation
will be transcendental in general.
\end{rk}

\subsection{Arbitrary lifts of $2$-plane congruences in 3D}

We showed in Proposition~\ref{p:normal} that any dLp can be made normal.
However, when $d=3$, the normal lift of a $2$-plane congruence $\Pi$ is not
unique. Instead, the rank $3$ distribution $\Delta=\hat{\pi}_*^{-1}(\Pi)\sub T\hat M_u$ 
has a unique Cauchy characteristic: a rank $1$ subbundle $\cch\sub\Delta$ with
$[\cch,\Delta]=\Delta$.  For a rank $2$ subbundle $\hat{\Pi}\sub\Delta$ the
normality condition $[\hat{\Pi},\hat{\Pi}]\sub\Delta$ implies that $\cch\sub
\hat{\Pi}$, but one generator of $\hat\Pi$ remains undetermined. In the case
of interest that $\Pi=W_0^\perp$ is characteristic for a quadric, an easy
computation shows that $\cch$ is spanned by the vector field
\[
\hat{W}=W_0+\z\p_\l,\qquad W_0:=V_0+2\l V_1+\l^2V_2,
\]
where, using the structure functions $c_{ij}^k=\te_k([V_i,V_j])$ of the
adapted frame $V_0,V_1,V_2$, we have
\begin{equation}\label{sigma}
\z=-c_{01}^2+\l(c_{01}^1-c_{02}^2)-\l^2(c_{12}^2-c_{02}^1+c_{01}^0)
+\l^3(c_{12}^1-c_{02}^0)-\l^4c_{12}^0.
\end{equation}
These formulae are compatible with representation $\z=m_0+2\l m_1+\l^2m_2$ for
the coefficients $m_i$ of the lifts of $V_i$ which we want to show can be
chosen quadratic in $\l$.

Without loss of generality we may write $\hat{\Pi}=\langle
\hat{W},\hat{U}\rangle$ with
\begin{equation}\label{hatU}
\hat{U}=W_1+\psi\p_\l, \qquad W_1=\tfrac12 (W_0)_\l=V_1+\l V_2.
\end{equation}
We also write $W_2=V_2=(W_1)_\l$. The nondegeneracy condition~\eqref{eq:ndg}
on $\Pi$ implies that $W_0,W_1,W_2$ form a ($\l$-dependent) frame for $TM_u$
and indeed
\[
W_0\odot W_2-W_1^2=V_0\odot V_2-V_1^2
\]
is the inverse metric to $g=4\te_0\te_2-\te_1^2$, which is nondegenerate and
independent of $\l$.

The Frobenius integrability condition $[\hat{\Pi},\hat{\Pi}]=\hat{\Pi}$ is
the condition that
\begin{equation*}
[\hat{W},\hat{U}]=[V_0,V_1]+\l[V_0,V_2]+\l^2[V_1,V_2]
+\sigma V_2-2\psi(V_1+\l V_2)+(\hat{W}(\psi)-\hat{U}(\sigma))\p_\l
\end{equation*}
is a section of $\hat\Pi$.  Identifying $V_1\equiv-\l V_2-\psi\p_\l$,
$V_0\equiv\l^2V_2+(2\l\psi-\sigma)\p_\l$ modulo $\hat{\Pi}$, and assuming that
the lift is normal, this reduces to $\normeq=0$, where
\begin{equation}\label{eqK}
\begin{split}
\normeq&:=(W_0+q_1+\sigma\p_\l)\psi+2\psi^2-\hat q_0,
\qquad \hat q_0=W_1\,\sigma+q_0\,\sigma,\\
q_1&=c_{02}^2-2c_{01}^1+\l(2c_{12}^2-3c_{02}^1+4c_{01}^0)
-\l^2(4c_{12}^1-5c_{02}^0)+6\l^3c_{12}^0,\\
q_0&=c_{01}^0+\l c_{02}^0+\l^2c_{12}^0.
\end{split}\end{equation}
Using the coefficients of the decomposition $[W_0,W_1]=\bar{c}_{01}^0
W_0+\bar{c}_{01}^1 W_1+\bar{c}_{01}^2 W_2$ we get
\begin{equation}\label{barsigma}
\sigma=-\bar{c}_{01}^2,\qquad q_1=-\bar{c}_{01}^1-\sigma_\l,\qquad
q_0=\bar{c}_{01}^0.
\end{equation}
Note that $\deg_\l \bar{c}_{01}^0=2$, $\deg_\l \bar{c}_{01}^1=3$ and $\deg_\l
\bar{c}_{01}^2=4$.

\begin{ex}[dKP] For the dKP equation~\eqref{dKP}, we have 
$g=\dd y^2-4\dd x\,\dd t+4u\,\dd x^2$, 
$\te=\dd t+\l\,\dd y+(\l^2-u)\,\dd x$ and
\[
\hat{W} = \p_x - 2\l\p_y + (\l^2+u)\p_t + (\l u_t-u_y)\p_\l,\qquad
\hat{U} = -\p_y + \l\p_t + \psi\p_\l,
\]
whence $\normeq=\psi_x-2\l\psi_y+(\l^2+u)\psi_t+(\l u_t-u_y)\psi_\l
-u_{yy}+2\l u_{yt}-\l^2u_{tt}-\psi u_t+2\psi^2$. In this case, via
the change of variables $\psi=\vp^{-1}+u_t$, the equation $\normeq=0$ 
\eqref{eqK} takes the linear inhomogeneous form
\begin{align}\label{Leq-dKP}
 &\cL_-(\vp)=2 \ \Leftrightarrow \ \cL_+(\vp^{-1})=-2\vp^{-2},\qquad\\
 &\text{where } \ \cL_\pm=\p_x-2\l\p_y+(\l^2+u)\p_t+(\l u_t-u_y)\p_\l \pm 3u_t.\nonumber
\end{align}
If we assume $\psi$ either local ($=$ differential) in $u$ or global ($=$
algebraic) in $\l$, then the only solution is $\vp^{-1}=0$, implying the
existence of a unique dLp of these types.

However, there exist solutions to~\eqref{Leq-dKP} which are non-algebraic in
$\l$ and nonlocal in $u$. Indeed for any Cauchy data $u|_{t=0}$ that is
non-algebraic in $\l$, we obtain such a solution.  In this way we obtain a
(characteristic but not projective or local) Lax pair that does not give rise
to an EW structure.  Moreover, there is no uniqueness for such Lax pairs.
\end{ex}

In the following Sections~\ref{Sec:32}--\ref{Sec:34} we deduce the cubic
behaviour of $\psi$ in $\l$ from the equation $\normeq=0$, 
a strengthened nondegeneracy condition, and the requirement that the dLp is local in $u$. 
This suffices to establish the projective property, and hence Theorem~\ref{t:proj}.

\subsection{Scalar PDEs in 3D}\label{Sec:32}

We first consider the case of a scalar differential equation $\E:F=0$ of order
$\ell$, i.e., one PDE~\eqref{F-ell} on one function $u$.  As before, we assume
that the characteristic variety $\chv^\E$ is a quadric, which implies that
$\ell$ is even and the symbol $F_{(\ell)}$ of the differential operator is a
power of a nondegenerate quadratic form: $\ell=2m$, $F_{(\ell)}=Q^m$ for some
$Q\in\Gamma(S^2TM_u)$ on $\E$.  (For a second order scalar PDE \eqref{F} we
get $m=1$.) Using the notation of the previous section, we have $Q=W_0\odot
W_2-W_1^2$.

The order of the conformal structure $c_F$ in $u$ satisfies
$k=\op{ord}(c_F)\leq\ell$, and the strict inequality is possible, for
instance, when $F$ is quasilinear (dKP is an example with $0=k<\ell=2$). Then
the frame and coframe $V_i$ and $\te_i$ can be chosen to have the same order
$k$ in $u$, while the structure functions $c_{\smash{ij}}^t$ and the coefficient
$\sigma$ in~\eqref{sigma} have order $\leq k+1$.

Let us suppose $\hat{\Pi}$ is a normal dLp for $\E$. We want to find an
$\E$-equivalent dLp which is projective. Since $\hat\Pi$ is normal, we may
suppose, as in the previous section, that its integrability condition is
$\normeq = 0$ with $\normeq$ given by~\eqref{eqK}. Hence by definition of a
dLp, $\mathfrak{e}=\Box F$ for some operator $\Box$ in total derivatives. If
$\psi$ has order $r\geq k+2$ then by taking the $(r+1)$-symbol of this
equation we obtain
\[
W_0\odot \psi_{(r)}=\mathfrak{e}_{(r+1)}=\Box_{(r+1-\ell)}\odot F_{(\ell)}
= \Box_{(r+1-2m)}\odot(W_0\odot W_2-W_1^2)^m,
\]
and hence conclude (since $W_0\odot W_2-W_1^2$ is indivisible by $W_0$) that
the symbol of $\psi$ is divisible by that of $F$. Therefore we can modify
$\psi$ off shell (fixed on shell) to obtain an $\E$-equivalent dLp in which
the new $\psi$ has order $<r$. By iterating this process, we may thus assume,
up to $\E$-equivalence, that $\psi$ has order $\leq k+1$ from the outset.

The $(k+2)$-symbol of $\normeq=\Box F$ now yields, using equation~\eqref{eqK},
the relation
\begin{equation}\label{zaq}
W_0\odot \psi_{(k+1)}-W_1\odot \sigma_{(k+1)}=R\odot(W_0\odot W_2-W_1^2)^m
\end{equation}
for a section $R\in\Gamma(S^{k-2m+2}TM_u)$ of the bundle of homogeneous degree
$k-2m+2$ polynomials on $T^*M_u$, i.e., $R=\sum_{|\tau|=k-2m+2}a_\tau W_\tau$,
where we let $W_\tau=W_{j_1}\odot\cdots\odot W_{j_t}$ for a multi-index
$\tau=(j_1\cdots j_t)$ of length $|\tau|=t$.  By modification of $\psi$ and
$\sigma$ off shell, we can bring this function to the form
\begin{equation}\label{nR}
R=(-1)^{m-1}\mu\, W_2^{k-2m+2}.
\end{equation}
Formula \eqref{zaq} then implies that
\begin{equation}\label{psisigma}
\sigma_{(k+1)}=W_1\odot R\odot Q^{m-1}+W_0\odot T,\qquad
\psi_{(k+1)}=W_2\odot R\odot Q^{m-1}+W_1\odot T
\end{equation}
for some $T\in\Gamma(S^kTM_u)$, and by the normalization~\eqref{nR}, the
coefficients for $R$ and $T$ are uniquely determined by independent components
of $\sigma$ and hence they are polynomial in $\l$. In particular, since
$\sigma$ is a quartic polynomial in $\l$, we conclude that $\mu\in
C^\infty(J^{k+1}M_u)$ is a polynomial in $\l$ with $\deg_\l\mu\leq3$.

Also, $T$ is a polynomial in $\l$ with $\deg_\l T\leq2$. Therefore,
$\psi_{(k+1)}$ is a cubic polynomial in $\l$.  Thus, there exists a function
$\psi_1=\psi_1(\p^{k+1}u,\l)$ with $\deg_\l\psi_1\leq3$ such that
$\psi_0:=\psi-\psi_1$ has order $\leq k$ in $u$.  Substituting
$\psi=\psi_1+\psi_0$ into the equation $\normeq=\Box F$ we get
\begin{equation}\label{psi0}
(W_0+\tilde{q}_1+\sigma\p_\l)\psi_0+2\psi_0^2=\tilde{q}_0,
\end{equation}
where $\tilde{q}_1=q_1+4\psi_1$, while $\tilde{q}_0$ is a expression of order
$k+1$ in $u$ that we do not write explicitly. However, it follows from
\eqref{barsigma} that $\tilde{q}_1,\tilde{q}_0$ are polynomial in $\l$ with
$\deg_\l\tilde{q}_1\leq3$, $\deg_\l\tilde{q}_0\leq 8$. We now want to show
that $\psi_0$ is also polynomial in $\l$.

In order to do this, it is convenient to carry out computations in the
nonholonomic $\l$-dependent frame $(W_i)_{i=0}^2$ rather than the holonomic
frame $(\p_{x^i})_{i=0}^2$ induced by local coordinates $(x^i)_{i=0}^2$ on
$M_u$. Let $(a_i^j)$ be the transition matrix between these frames and
$(b_i^j)$ be its inverse, i.e., $W_i=a_i^j\p_{x^j}$ and $\p_{x^i}=b_i^jW_j$
(summation convention). These vector fields induce vertical vector fields
$\mathbb{D}_{W_\tau}$ on jets via the formulae
$\mathbb{D}_{W_\tau}=b_{j_1}^{i_1}\cdots b_{j_t}^{i_t}\p_{u_{i_1\cdots i_t}}$,
where $\tau=(j_1,\ldots j_t)$. If $\xi$ is a function on the jet bundle (i.e.,
a differential operator) with order $t$ symbol $\xi_{(t)}$, then
$\mathbb{D}_{W_\tau}\xi$, with $|\tau|=t$, is the the coefficient $\xi_t^\tau$
of $\xi_{(t)}$ in the decomposition $\xi_{(t)}=\xi_t^{\rho} W_\rho$ (summation
over multi-indices $\rho$ with $|\rho|=t$).

Next, since the dual co-frame to $(W_0,W_1,W_2)$ is
$(\frac12\te_{\l\l},-\te_\l,\te)$, the coefficients on the right hand sides of
identities \eqref{barsigma} may be written
\[
\bar{c}_{01}^2=\te([W_0,W_1]),\ \ \bar{c}_{01}^1=-\te_\l([W_0,W_1]),\ \
\bar{c}_{01}^0=\tfrac12\te_{\l\l}([W_0,W_1]).
\]
This leads to
\begin{equation}\label{mudef}
\mu=\mathbb{D}_{W_1^{2m-1}W_2^{k-2m+2}}(\sigma)
=\mathbb{D}_{W_1^{2m-2}W_2^{k-2m+3}}(\psi)
=\te\bigl(\mathbb{D}_{W_1^{2m-2}W_2^{k-2m+2}}(W_0)\bigr).
\end{equation}
Identity \eqref{psisigma} yields $\te(\mathbb{D}_{W_1^rW_2^{k-r}}W_0)=0$
unless $r=2m-2$.  Note that by \eqref{psisigma} we have $\sigma_{(k+1)} =
\mu\,W_1^{2m-1}\odot W_2^{k-2m+2}\op{mod}W_0\odot S^kTM_u$, and this
decomposition can be refined:
\begin{align*}
\sigma_{(k+1)} &= \mu\, W_1^{2m-1}\odot W_2^{k-2m+2}
+\gamma\, W_0\odot W_1^{2m-2}\odot W_2^{k-2m+2}+\cdots\\
\psi_{(k+1)} &= \mu\, W_1^{2m-2}\odot W_2^{k-2m+3}
+\gamma\, W_1^{2m-1}\odot W_2^{k-2m+2}+\cdots,
\end{align*}
where $\gamma
=\mathbb{D}_{W_1^{2m-2}W_2^{k-2m+2}}(T)
=\te\bigl(\mathbb{D}_{W_0W_1^{2m-3}W_2^{k-2m+2}}(W_0)\bigr)
-\te\bigl(\mathbb{D}_{W_1^{2m-2}W_2^{k-2m+2}}(W_1)\bigr)$
and by dots we mean all terms with other $W_\tau$ that are irrelevant for the
computation.  Consequently,
\[
(\sigma_\l)_{(k+1)} = (2m-1)\mu\, W_1^{2m-2}\odot W_2^{k-2m+3}
+(\mu_\l+2\gamma)\, W_1^{2m-1}\odot W_2^{k-2m+2}+\cdots
\]
and since $\sigma_\l=(\te[W_1,W_0])_\l=\te_\l[W_1,W_0]+\te[W_2,W_0]
=\bar{c}_{01}^1-\te[W_0,W_2]$ we get
\[
(\bar{c}_{01}^1)_{(k+1)}=(\sigma_\l)_{(k+1)}-\mu\, W_1^{2m-2}\odot W_2^{k-2m+3}
\op{mod} W_0\odot S^kTM_u.
\]
Thus from \eqref{barsigma} we get the following expression for the $(k+1)$-symbol
\begin{align*}
(\tilde{q}_1)_{(k+1)}&=
-2(\sigma_\l)_{(k+1)}+5\mu\, W_1^{2m-2}\odot W_2^{k-2m+3}
+4\gamma\, W_1^{2m-1}\odot W_2^{k-2m+2}+\cdots\\
&=(7-4m)\mu\, W_1^{2m-2}\odot W_2^{k-2m+3}
-2\mu_\l\, W_1^{2m-1}\odot W_2^{k-2m+2}+\cdots
\end{align*}
Taking now $(k+1)$-symbol of~\eqref{psi0}, we get
\[
W_0\odot(\psi_0)_{(k)}+(\tilde{q}_1)_{(k+1)}\,\psi_0+\sigma_{(k+1)}\,(\psi_0)_\l
=(\tilde{q}_0)_{(k+1)}.
\]
Denoting $(\tilde{q}_0)_{(k+1)}=\kappa_0\, W_1^{2m-2}\odot W_2^{k-2m+3}
+\kappa_1\, W_1^{2m-1}\odot W_2^{k-2m+2}+\cdots$ and extracting the
coefficients at the indicated terms (which are unchanged by $\E$-equivalence)
we obtain the following system
\begin{align*}
(7-4m)\mu\psi_0&=\kappa_0,&
\mu\p_\l\psi_0-2\mu_\l\psi_0&=\kappa_1.
\end{align*}
All coefficients of this linear system on $\psi_0$ are polynomials in $\l$.
We assume $\mu\neq0$ (this condition will be discussed in the next section).
Then the first equation uniquely determines $\psi_0$. Moreover, $\mu$ divides
$\kappa_0$ as a polynomial in $\l$ because otherwise $\psi_0$ is a proper
rational function and then the second equation, written as
$(\psi_0\mu^{-2})_\l=\kappa_1\mu^{-3}$, yields a contradiction.

Thus $\psi_0$, and hence also $\psi$, are polynomials in $\l$, and
$\deg_\l\psi\leq5$.  Since the parameter $\l$ is manifestly projective, a
projective change should not destroy the polynomial property. Using the
special projective transformation $\l\mapsto\l^{-1}$ (or a similar projective
transformation arbitrarily close to the identity), we conclude that in fact
$\deg_\l\psi\leq3$.

Moreover, to satisfy smoothness in $\l$, the function $\psi$ should be
compatible with $\z$ in the sense that $\psi=m_1+\l m_2$, $\z=m_0+2\l m_1+\l^2
m_2$ for some $\l$-quadrics $m_i$.

\subsection{Nondegeneracy for scalar and vector equations}\label{Sec:33}

The $\l$-dependent quantity $\mu$ introduced in \eqref{mudef} characterizes
the extent to which the Lax integrability condition depends on the equation
$\E$. If this condition does not involve $F$ and its derivatives, the dLp is
trivial (holds off shell). We require that the equation shows on the level of
the top symbol, i.e., $(k+1)$-jet, or equivalently that $\mu\neq0$.  We first
observe that this condition is invariant under admissible transformations of
$\hat M_u$ as a $\PP^1$-bundle over $M_u$.

\begin{prop}\label{relinv}
The scalar quantity $\mu$ is a relative differential invariant, i.e.,
transforms by a nonvanishing scalar multiple under admissible transformations.
\end{prop}
\begin{proof}
The admissible transformations of $\hat M_u$ have the form
$(\bx,\l)\mapsto(\Phi(\bx),\Psi(\bx,\l))$, where $\Phi$ is a conformal
transformation of $(M_u,[g])$ and
$\Psi(\bx,\l)=\frac{a(\bx)+b(\bx)\l}{c(\bx)+d(\bx)\l}$ is a parametric
M\"obius transformation.  These preserve the algebraic behaviour of the dLp
$\hat{\Pi}$, and a straightforward computation shows they scale $\mu$ by a
nonvanishing scalar multiple.

Alternatively, using the framework and normalizations of Section~\ref{Sec:32},
$\sigma$ given by \eqref{sigma} is independent of the adapted frame up to
scale and the leading coefficient of its symbol
$\mathbb{D}_{W_1^{2m-1}W_2^{k-2m+2}}(\sigma)$ is a relative invariant, as
required.
\end{proof}

Let us now give the vector version, recalling first the set-up. In this case
$F\colon J^\ell(M,\V)\to\W$ is a determined (nonlinear) differential operator
of order $\ell$ on sections $\bu$ of a fibre bundle $\V$ over $M_u$ with
values in a rank $s$ vector bundle $\W$. We assume, for simplicity, that $\V$
is also a vector bundle of rank $s$ so that we can identify the vertical
bundle $T_\bu^{\rm v}\V$ along a section $\bu$ with $\V$.  Locally, in
coordinates, $F$ has components $F^i$ that are scalar differential operators
of order $\ell$ on vector-function $\bu=(u^j)$ of $\bx$, where
$i,j\in\{1,\ldots s\}$.

The symbol of $F$ at $\bu=(u^i)$ is a map $F_{(\ell)}\colon S^\ell
T^*M_\bu\ot\V\to\W$ that we identify with an $s\times s$ matrix $\bF=(F_i^j)$,
whose coefficients are polynomials of degree $\ell$ on $T^*M_\bu$. Similarly,
the symbol of a scalar differential operator $\varphi$ can be identified with
a column in components. The characteristic variety $\chv^\E$ of $\E$ is a
quadric if $\det(\bF)=Q^m$ as before.

The setup of the previous section extends, and
$\mu=\te\bigl(\mathbb{D}_{W_1^{2m-2}W_2^{k-2m+2}}(W_0)\bigr)$ is a section
of the bundle $\V$ over $M_u$ for solutions $u$ of the vector version of
\eqref{F-ell}. The following statement is proved similarly to Proposition
\ref{relinv}.

\begin{prop}
The section $\mu$ is a relative differential invariant, i.e.,
under admissible transformations it is mapped to another section 
related to $\mu$ by an automorphism of the bundle $\V$.
Hence the \textup(non-\textup)vanishing of $\mu$ is an invariant property.
\end{prop}

Note that $\mu$ depends not on a lift or dLp but only on the equation $\E:F=0$
itself.

 \begin{dfn}\label{Dnd}
In 3D the equation is called nondegenerate (and its dLp nontrivial) if the
relative invariant $\mu$ is nonzero (identically in $\l$).
 \end{dfn}

This condition is trivially satisfied if the conformal structure $c_F$ has
zero order in $u$, as happens in the dKP case. It can be proved for several
classes of PDE in 3D (with $\op{ord}_u(c_F)>0$), and we do not know of any
integrable equation violating this condition.

\begin{rk}\label{rem:MS-nondegen}
In fact, the Manakov--Santini equation, which by \cite{DFK} is the master
equation for EW geometry, is nondegenerate in the sense of this definition.
We check this for the modified version, with $W_0$ given by~\eqref{eq:W0MS}.
Since order of the conformal structure in this formalism is $k=0$, and also
$m=1$, we compute the symbol of $W_0$ by $(a,b)$ as $(\p_t,\p_y)$ and applying
$\te$ given by~\eqref{thla} we get $\mu=(\l,1)\neq0$.  Thus the MS equation is
nondegenerate and we adopt this condition for our main result.
\end{rk}

\subsection{Proof of Theorem~\ref{t:proj}}\label{Sec:34}
Any nondegenerate dLp $\hat\Pi$ is $\E$-equivalent to a normal dLp by
Proposition~\ref{p:normal}. When $d=4$, the normal dLp has the projective
property, while for $d=3$, the main task is to show that, up to
$\E$-equivalence, we may assume that $\psi$ in~\eqref{hatU} is cubic in
$\l$. The proof almost directly generalizes the scalar version of
Section~\ref{Sec:32}, so we only indicate important differences on each step.
\begin{numlist}
\item We begin with equation $\normeq=\Box F$ and, as before, by an off shell
  modification can arrange that $\op{ord}\psi\leq k+1$, where $k$ is the order
  of the conformal structure $c_F$.  Then its $(k+2)$-symbol and \eqref{eqK} yield
  the following matrix equation
\begin{equation}\label{matrix-psisigma}
   \begin{bmatrix} \psi^1 & \sigma^1 \\  \vdots & \vdots \\ \psi^s & \sigma^s
   \end{bmatrix}\odot
\begin{bmatrix} W_0 \\ -W_1 \end{bmatrix}=
\begin{bmatrix} F^1_1 & \cdots & F^1_s \\ \vdots & \ddots & \vdots \\
  F^s_1 & \cdots & F^s_s \end{bmatrix}
\odot \begin{bmatrix} R^1 \\ \cdots \\ R^s \end{bmatrix}
\end{equation}
where $\psi^i$ are components of the symbol $\psi_{(k+1)}$ and similarly for
$\sigma$, and where $R^i$ are the symbols of some operators in total
derivatives. Multiplying this equation by the adjugate matrix $\op{adj}(\bF)$
(which satisfies $\op{adj}(\bF)\bF=\det(\bF)I=Q^m I$) and denoting the rows of
the resulting left-hand side matrix $[\tilde\psi^i\ \tilde\sigma^i]$ we get
the equations
\[
W_0\odot\tilde\psi^i-W_1\odot\tilde\sigma^i=R_i\odot(W_0\odot W_2-W_1^2)^m,
\qquad i\in \{1,\ldots s\},
\]
from which we obtain a vector analogue of equation \eqref{psisigma} for each
component $i\in\{1,\ldots s\}$. Moreover, we can obtain normalization
analogues $R_i=(-1)^{m-1}\mu_i\, W_2^{k-2m+2}$ of \eqref{nR}.  This implies
that $\tilde\psi^i$ and hence $\psi^i$ can be chosen polynomial in $\l$,
moreover $\deg_\l\psi^i\leq3$.

\item Thus there exist a decomposition $\psi=\psi_1+\psi_0$, where $\psi_1$ is
  at most cubic in $\l$ and has $(k+1)$-symbol $(\psi^i)$ at $\bu=(u^j)$,
  while $\psi_0$ has order $\leq k$.  Substituting this into the constraint
  $\mathfrak{e}=\Box F$ we obtain a vector analogue of equation \eqref{psi0}.
  Taking its $(k+1)$-symbol and applying $\op{adj}(\bF)$ again gives
\[
\tilde{q}_1{\!}^i\psi_0+\tilde{\sigma}^i\p_\l\psi_0=\tilde{q}_0{\!}^i,\qquad
i\in\{1,\ldots s\}.
\]
If $\mu=(\mu_1,\ldots\mu_s)$ is nonzero (identically in $\l$), we conclude by
the same argument that $\psi_0$ is polynomial in $\l$ with
$\deg\psi_0\leq5$. Hence $\psi$ is a polynomial in $\l$ with $\deg\psi\leq5$.

\item Finally, if the coefficients of $\psi$ at $\l^4$ or $\l^5$ are nonzero,
  then a M\"obius transformation
  $\l\mapsto\frac{a(\bx)+b(\bx)\l}{c(\bx)+d(\bx)\l}$ arbitrary close to the
  identity maps the system of vector fields $\hat{W}$ and $\hat{U}=V_1+\l
  V_2+\psi\p_\l$ (after taking a proper linear combination and clearing
  denominators) to a system of the same form with a new $\psi$ of higher
  degree. Thus we must have $\deg_\l\psi\leq3$, which is a projectively invariant
  property.
\end{numlist}
In addition to $\hat{U}$ the vector field $\hat{W}-\l\hat{U}=V_0+\l
V_1+(\sigma-\l\psi)\p_\l$ have degree $\leq 3$ in $\l$. Indeed, under a change
of the adapted frame $(V_0,V_1,V_2)$ and a projective change of parameter $\l$
this field becomes of the form $\hat{U}$ and so the claim follows from
(i)--(iii).

Theorem \ref{t:proj} is now immediate. By Lemma \ref{lem3D} normal lifts with
this projective property are bijective with Weyl connections for $d=3$, while
for $d=4$ the normal lift is unique by Lemma \ref{lem4D}.  Thus for $d=3$ or
$d=4$, the standard Lax pair of $(M_u,c_F,\nabla)$ or $(M_u,c_F)$ is
$\E$-equivalent to $\hat\Pi$.  \qed

\section{Applications and generalizations}\label{Sec:4}

\subsection{Pseudopotentials}\label{Sec:42}

In this paper we have defined dispersionless integrable systems using a Lax
pair of vector fields. In 3D, an alternative approach relies instead on
pseudopotentials or nonlinear dispersionless Lax pairs, cf.~\cite{Za,OS,FK}.

\begin{dfn}
A \emph{pseudopotential} for a PDE $F=0$ is a function $S\colon M_u\to \R$
whose derivative $\dd S$ satisfies an overdetermined system of two equations that are compatible on shell, i.e., when $F(j^\ell u)=0$.
\end{dfn}

Locally, in coordinates $(x,y,t)$, we may write these equations as
$S_x=A(S_t)$ and $S_y=B(S_t)$ where $A$ and $B$ also depend on $(x,y,t)$. If
they depend on $(x,y,t)$ (only or also) through a section $v$ of a vector
bundle over $M_u$, and the integrability condition $\p_y(A(S_t))=\p_x(B(S_t))$
is required to hold identically in $S_t$, we obtain a PDE system on $v$.
Dispersionless integrable systems are often defined as those determined PDEs
arising in this way.

More invariantly, the two equations determine a codimension two (hence
$4$-dimensional) submanifold $N$ of the cotangent bundle $T^*M_u$ and $S$ is a
pseudopotential with respect to these equations if $\dd S$ takes values in
$N$. The integrability condition means that $N$ is coisotropic for the
canonical symplectic form $\Omega$ on $T^*M_u$. Here we recall that
$\Omega=\dd\tau$ where $\tau$ is the tautological $1$-form on $T^*M_u$ (with
$\beta^*\tau=\beta$ for any $1$-form $\beta$ on $M_u$). The coisotropic
condition means that the pullback of $\Omega$ to $N$ has rank two, hence a
$2$-dimensional radical (or kernel).

Locally $N$ is a fibre bundle over $M_u$ and we may take $\l=S_t$ as a fibre
coordinate in the above explicit formulation. Thus
\begin{align*}
\p_y(A(S_t)) &= A_y + A_\l S_{tx} = A_y + A_\l B_t + A_\l B_\l S_{tt},\\
\p_x(B(S_t)) &= B_x + A_\l S_{ty} = B_x + A_t B_\l + A_\l B_\l S_{tt},
\end{align*}
and so the integrability condition is
\begin{equation}\label{poisson}
A_y - B_x = \{A,B\}_P := A_t B_\l - A_\l B_t,
\end{equation}
where $\{A,B\}_P$ is the Poisson bracket of $A$ and $B$ with respect to the
Poisson structure $P:=\p_t\wedge\p_\l$; equivalently, the vector fields $\p_x
+ P(\dd A)$ and $\p_y + P(\dd B)$ commute, where $P(\dd A)= A_t \p_\l -
A_\l\p_t$ and $P(\dd B)=B_t \p_\l-B_\l\p_t$ are the hamiltonian vector fields
associated to $A$ and $B$ by the Poisson structure $P$ (cf. e.g.~\cite{DGS}).

Alternatively, if~\eqref{poisson} holds, then the pullback of $\Omega$ to $N$
is
\begin{multline*}
\dd\l \wedge \dd t + (A_\l \dd\l+ A_t\dd t) \wedge \dd x
+ (B_\l \dd\l+ B_t\dd t) \wedge \dd y - (A_t B_\l-A_\l B_t) \dd x\wedge \dd y\\
= (\dd\l - A_t \dd x - B_t \dd y)\wedge (\dd t + A_\l \dd x + B_\l \dd y)
\end{multline*}
and its radical is the dLp spanned by $\p_x + P(\dd A)$ and $\p_y + P(\dd B)$.

Conversely, let $\hat\Pi$ be a dLp on $\hat\pi\colon\hat M_u \to M_u$. On
shell, $\hat\Pi$ is integrable and so $\hat M$ fibres locally over a
minitwistor space $\tw$~\cite{Hi}. At least locally $\tw$ admits a
nondegenerate (and necessarily closed) $2$-form (such as $\dd z_1\wedge\dd
z_2$ in local coordinates); this then pulls back to a closed $2$-form $\omega$
on $\hat M_u$ with radical $\hat\Pi$. We may therefore write (locally, on
shell) $\omega=\dd\alpha$ for a $1$-form $\alpha$ on $\hat M_u$, which we may
assume vanishes on the fibres of $\hat M_u$ over $M_u$; hence we may write
$\alpha=(\Id,\tilde\alpha)\circ\hat\pi_*$ for a section $(\Id,\tilde\alpha)$
of $\hat\pi^*T^*M_u = \{(\hat p,\xi)\in \hat M_u\times T^*M_u\,|\, \xi\in
T^*_{\hat\pi(\hat p)} M_u\}$. Then $\tilde\alpha\colon \hat M_u\to T^*M_u$ is
an immersion whose image is coisotropic, since $\tilde\alpha^*\tau=\alpha$ and
so $\tilde\alpha^*\Omega=\dd\alpha=\omega$ has rank two with radical
$\hat\Pi$.

In order to do this off shell, we have to work modulo the PDE system.
However, the construction of the coisotropic immersion $\tilde\alpha$ from a
dLp requires integration, and so it may be necessary to pass to a covering
system.

\begin{ex}[dKP]
We illustrate this with the well-known example of the dKP equation~\eqref{dKP}
$(u_x + u u_t)_t = u_{yy}$ with dLp~\eqref{dKPdlp}. We must now find a
function $f$ so that $\omega=f\eta\wedge\te$ is closed modulo the equation,
and then a $1$-form $\alpha$ such that $\dd\alpha=\omega$ modulo the
equation. For the first step, it happens in this case that $f=1$ works. For
the second, setting
\[
\alpha= (\tfrac13\l^3 - u\l - v) \dd x + (\tfrac12\l^2 - u)\dd y + \l\,\dd t,
\]
we have that $\dd\alpha=\eta\wedge\te$ modulo the covering system $v_t=u_y$
and $v_y=u_x + u u_t$. Thus the pseudopotential system is
$S_x=\frac13 S_t^3 - u S_t - v$ and $S_y=\frac12 S_t^2 - u$.

Note that the above nonlocality (usage of $v$) may be avoided by using the
potential form $u_{xt}+u_t u_{tt}-u_{yy}=0$ of dKP. In this case the
pseudopotential $S$ is given by the equations: $S_x=\l^3/3-u_t\l-u_y$,
$S_y=\l^2/2-u_t$, $S_t=\l$. In both cases the parameter $\l$ is aligned to the
Lax pair in the sense that it is the projective parameter on the
correspondence bundle $\hat M_u\to M_u$. This is no longer so with
Manakov--Santini system \eqref{MS}.
\end{ex}

\begin{ex}[MS]
The MS system does admit a pseudopotential formulation; however it is neither
local in $u,v$ nor rational in $\l$. The system
\begin{gather}
\vphantom{\frac{a}{a}}
\hspace{-5pt} 
\sigma (R_x - P_t) = (u_t\l+u_y) (Q_x -P_y), \ \
\sigma (R_y - Q_t) = u_t (Q_x -P_y),\label{covMS1}\\
\hspace{-5pt} \sigma P_\l  = (\l ^2 +v_t\l -u + v_y) (Q_x -P_y), \
\sigma Q_\l  = (\l +v_t) (Q_x -P_y), \
\sigma R_\l=(Q_x -P_y)\label{covMS2}
\vphantom{\frac{a}{a}}
\end{gather}
with $\sigma=u_y\l +u u_t -u_t v_y +u_y v_t$, is a differential covering,
meaning it is compatible modulo MS. Here the last three equations \eqref{covMS2} determine
the behaviour in the spectral parameter $\l$, while the first two equations \eqref{covMS1} 
yield a pseudopotential $S$ via the system $S_x=P$, $S_y=Q$, $S_t=R$.  Indeed,
one can verify that the differential $\omega=\dd\alpha$ of the 1-form
$\alpha=P \dd x+Q \dd y+ R \dd t$ on $\hat M_u$ satisfies
$\hat{X}\inm\omega=\hat{Y}\inm\omega=0$ modulo MS and
\eqref{covMS1}--\eqref{covMS2}, where $\hat{X}=\tilde{X}|_{\tilde{\l}=\l}$,
$\hat{Y}=\tilde{Y}|_{\tilde{\l}=\l}$ in terms of formula \eqref{LP-MS} are
vector fields on $\hat M_u$ forming the Lax pair (with parameter $\l$
projective).
\end{ex}

\subsection{Twistor interpretation via contact coverings}

To relate the pseudopotential formulation more closely to the dLp formulation,
we focus on the first order quasilinear system for sections of
$\hat\pi\colon\hat M_u\to M_u$ which correspond to hypersurfaces in the
twistor space. We refer to this PDE system as a \emph{contact covering} of
$\E$ because the equation it defines is a codimension $2$ submanifold $\ccov$ of
$J^1\hat\pi$ (the bundle of $1$-jets of sections of $\hat\pi$), which is a
contact manifold.

This viewpoint gives an alternative way to understand why contact coverings
are equivalent to dLps. For this, let $\alpha$ be a contact form on
$J^1\hat\pi$ representing the standard contact structure and let
$\alpha_\ccov$ be its restriction on $\ccov$. Then for
$\omega_\ccov=\dd\alpha_\ccov$ we have: $\alpha_\ccov\we\omega_\ccov^{d-1}=0$
on shell, but $\alpha_\ccov\we\omega_\ccov^{d-2}\not\equiv0$, which implies
that $\alpha_\ccov$ has a $2$-dimensional radical $\hat{\Pi}\sub T\ccov$:
$\xi\inm\alpha_\ccov=0$ and $\xi\inm\omega_\ccov=0$ for all $\x\in\hat{\Pi}$.

If $\ccov$ is quasilinear, as we require in the definition of a contact
covering, then $\hat{\Pi}$ is projectible along the fibres
$\pi_{1,0}\colon\ccov\to J^0\hat{\pi}=\hat M_u$ and so it induces a pushforward
distribution of rank $2$ in $T\hat M_u$, which is a dLp in our formalism. This
is how a nonlinear covering induces a linear one, and the inverse relation is
given by a lift.

We summarize the observed relations into the following diagram, intertwining
the twistor and jet concepts:
\[
\xymatrix{
& (J^1\hat\pi)^{2d+1} \ar@{-->}[ld] \ar[dd]_{\pi_1} \ar[rd]^{\pi_{1,0}} &&
\ccov^{2d-1} \ar@{_{(}->}[ll] \ar@{.>}[ld] \ar[rd]^{\hat{\Pi}^2} \\
(T^*M_u)^{2d} \ar[rd] && J^0\hat{\pi}=\hat M_u^{d+1} \ar[ld]_{\PP^1} \ar[rd]^{\hat{\Pi}^2} &&
\PP(T^*\tw)^{2d-3} \ar[ld]_{\PP^{d-2}} \\
& M_u^d && \tw^{d-1}
}
\]
Here the dotted arrow is the restriction of the jet-projection to the contact
covering $\ccov$, arrows labelled by $\hat{\Pi}$ are (local) quotients by the
corresponding foliations, and $\tw^{d-1}$ is the mini-twistor or the twistor
space for $d=3$ or $d=4$ respectively.

The dashed arrow is well-defined locally, when a local coordinate (spectral
parameter) on the fibre $\hat M_u$ is chosen, but it may fail to exist
globally with respect to the spectral parameter $\l$ and locally with respect
to the dependent variable $u$. For $d=3$, this is precisely the theory of
pseudopotentials as discussed in Section~\ref{Sec:42}.  In this case, the
space $\PP(T^*\tw)$ is the real Penrose twistor space (projecting to the
Hitchin mini-twistor space with fibre $\PP^1$) that embeds into the complex
twistor space $\tw_\C$, which is the complexification of $\tw$ of the case
$d=4$, via a conformal Killing reduction~\cite{JT}.  We leave to the reader a
specification of relations between different real forms (signatures of the
conformal structure---related by Wick rotations in physics language).

When $d=4$ an analogue of the theory of pseudopotentials has been developed
in~\cite{Se}. Geometrically, this involves making a further projection to
$\PP(T^*M_u)$ on the left hand side of the above diagram (this is why the Lax 
pairs of \cite{Se} are homogeneous in $\p\psi$ for the covering function $\psi$).
This is a $7$-dimensional contact manifold, so two equations suffice to define 
a $5$-dimensional submanifold $\hat M_u$. In this formalism the Lax pair is
given by contact hamiltonian vector fields.

\subsection{Extensions of the theory}\label{S:5}

First, as noted in the introduction, in 2D, the theory of dispersionless Lax
pairs is vacuous, essentially because there is only one $2$-plane congruence.
However, if we relax the assumption that the Lax pair is transverse to the
fibres of $\hat M_u$ over $M_u$, this objection evaporates. The characteristic
condition means that at points of tangency, the projection of the Lax
distribution is a characteristic direction. In particular, when the
characteristic variety is a quadric (two points), we expect two points of
tangency, with the background given
by the spinor-vortex equations~\cite{C1}.

Secondly, it would be nice to relax the requirement that the PDE system
$F\colon J^\ell(M,\V)\to \W$ determined in the sense that
$\op{rank}(\W)=\op{rank}(\V)$. The theory in this paper should at least extend
to (formally) overdetermined systems ($\op{rank}(\W)\geq\op{rank}(\V)$) which
are compatible, so that the characteristic variety is a hypersurface. We would
then need to use the compatibility conditions to generalize Theorem~\ref{t:char}. 
 
For truly overdetermined systems, with characteristic variety of higher 
codimension, it would be necessary also to replace Lax pairs by Lax distributions 
of higher rank. Recently the characteristic property was confirmed in \cite{FK2} for paraconformal structures
generalizing EW structures to higher dimension, and we suggest that
it applies universally.

Finally, with the latter idea, the restriction to dimensions $d=3,4$ can be
relaxed. This would extend the framework of integrability via geometry to a wider context.

\subsection*{Acknowledgements}

We are grateful to Jenya Ferapontov for helpful discussions. In particular, he
pointed out that dispersionless Lax pairs for PDEs of Hirota type are always
characteristic, suggesting that this might be true more generally, and later
drew our attention to difficulties in establishing the projective property for
Lax pairs in 3D.


\end{document}